\documentclass{article}

\usepackage{amssymb, amsmath, enumerate, wrapfig, mathrsfs, amsthm, fullpage}
\usepackage{cite}
 \usepackage{graphicx, caption, subcaption, float, hyperref, sidecap}
 \usepackage{color} 
 \usepackage{tikz} 
 
  \newcommand{\R}{\mathbb{R}}
  \newcommand{\N}{\mathbb{N}}
   \newcommand{\Z}{\mathbb{Z}}
  \newcommand{\HR}{{}^* \R}
  \newcommand{\NS}{\mathscr}
\newcommand{\cart} {
\left(
\begin{array}{ccc}
a & 0 & 0 \\
0 & b & 0 \\
0 & 0 & \frac{1}{ab} \end{array}
\right)}
 
\newtheorem{theorem}{Theorem}
\newtheorem{lemma}[theorem]{Lemma}
\newtheorem{proposition}[theorem]{Proposition}
\newtheorem{corollary}[theorem]{Corollary}
\newtheorem{definition}[theorem]{Definition}
\newtheorem{remark}[theorem]{Remark}

\title{Conjugacy Limits of the Diagonal Cartan Subgroup in $SL_3( \mathbb{R})$ }  \date{}
\author{Arielle Leitner}
\date{\today}

\begin{document} 

\maketitle  
\let\thefootnote\relax\footnote{Department of Mathematics, University of California, Santa Barbara, CA 93106  \hfill Email: \href{mailto:aleitner@math.ucsb.edu}{aleitner@math.ucsb.edu}} %\email{aleitner@math.ucsb.edu} 

\begin{abstract}
 A \emph{conjugacy limit group} is the limit of a sequence of conjugates of the positive diagonal Cartan subgroup, $C \leq SL_3(\R).$  We prove a variant of a theorem of Haettel, and show that up to conjugacy, $C$ has 5 possible conjugacy limit groups.  Each conjugacy limit group is determined by a  nonstandard triangle.  We give a criterion for a sequence of conjugates of $C$ to converge to each of the 5 conjugacy limit groups.   MSC 57S25.
\end{abstract}

\section{Motivation and the Main Theorems}%%%%%%%%%%%%%%%%%%%%%%%

We study geometries in the sense of Klein:  a $(G,X)$ geometry is a Lie group $G$ acting transitively on a manifold $X$. The idea of continuously deforming one kind of geometry into another appears in many areas of mathematics and physics.   For the most basic intuition, we take the example given in \cite{Cooper}, that a sequence of spheres tangent to a plane, with increasing radius, will limit to the tangent plane in the Hausdorff topology on closed sets.   
Another viewpoint mentioned in \cite{Cooper} is that of Lie algebras.  Homogeneous spaces have groups of isometries that are Lie groups, which have an associated Lie algebra determined by the Lie bracket $[\cdot, \cdot] :  \mathfrak{g} \times \mathfrak{g} \to \mathfrak{g}$. This is a bilinear map, and is determined by the values it takes on a basis, and hence by \emph{structure constants}.  We may continuously change the structure constants as long as $[\cdot, \cdot]$ still determines a Lie algebra.  This is related to theory of In\"{o}n\"{u}-Wigner contractions in physics, see \cite{Burde} and \cite{IW}. 

%Finally, \cite{Cooper} discusses study transitions between Thurston geometries- the eight geometries that play a major role in classifying compact three manifolds.  These are (almost) subgeometries of real projective geometry, \cite{Cooper}, and we can study geometric transitions in this context, as paths of conjugacies. 

Conjugacy limits of $C$ are related to deformations of convex $\R P^2$ structures.   Any discrete $\Z \oplus \Z$ subgroup of $C$ gives a convex projective structure on the torus.   Choi and Goldman give two examples of such convex projective structures in \cite{CG}.   There are two kinds of deformations: deforming lattices in $C$, or deforming $C$ inside $SL_3(\R)$.  This paper concerns the latter.

This paper classifies conjugacy limits of the positive diagonal Cartan 
subgroup $C \leq SL_3 ( \R)$.  %A  Cartan subalgebra of a Lie algebra is a nilpotent, self-normalizing Lie subalgebra.  A Cartan subgroup of a Lie group is a subgroup that corresponds to a Cartan subalgebra.  %In $SL_n(\R) $, Cartan subgroups are the maximal connected abelian semi-simple Lie subgroups, or maximal tori.  In particular, one Cartan subgroup in $SL_3 ( \R)$ is the subgroup of diagonal matrices, a 2-dimensional abelian group: 
%$$C:= \Big \{ \cart : a ,b \in \R \text{ and } a,b >0  \Big \} ,$$  which is isomorphic to $\R ^2$. 
%This paper is focused on classifying the conjugacy limits of the diagonal Cartan subgroup in $SL_3 ( \R)$. 
Let $G$ be a Lie group and $H$ a closed subgroup.  A sequence of subgroups $H_n$ of $G$ \emph{converges} to $H$ if the following two conditions are satisfied:\\
 (a) For every $h \in H$ there is a sequence $h_n\in H_n$ converging to $h$\\ 
 (b) For every sequence $h_n\in H_n$, if there is a subsequence which converges to $h$, then $h \in H$.   \\
 A subgroup $L \leq G$ is a  \emph{conjugacy limit} of a subgroup $H \leq G$ if there is a sequence of elements $P_n \in G$ such that $P_n H P_n ^{-1}$ converges to $L$. (See \cite{Cooper} definition 2.5).
% Let $C$ be the identity component of the diagonal Cartan subgroup in $SL_3(\R)$. 
The identity component of a conjugate of $C$ is the stabilizer of the vertices of a triangle in $\R P^2$.   We will show in theorem \ref{maintri} that a conjugacy limit of $C$ is characterized in terms of the stabilizer of a \emph{configuration}, which is a set whose elements are points and lines in $\R P^2$. %the stabilizer of a \emph{degenerate triangle}, the limit of a projective triangle under a sequence of projective transformations. 
 %We show that a conjugacy limit of $C$ is the {\em shadow} of the stabilizer of a {\em nonstandard triangle}.

%A projective triangle is a subset of $\R P ^2$ containing three non-collinear points and the lines between each pair of points.   A degenerate triangle is the limit of a projective triangle under a sequence of projective transformations.   We first state the main theorem, and then give the necessary background on the hyperreals, $\HR$. 
%We establish a 1-1 correspondence conjugates of 2-dimensional abelian subgroups of $SL_3 (\R)$ that maximally preserve equivalence classes of degenerate triangles.  Thus we will discuss conjugacy classes of 2-dimensional abelian groups and degenerate triangles interchangeably. 

% \begin{prop}\label{3 pts} The set of conjugates of the Cartan subgroup in $SL_n ( \R)$ is in 1-1 correspondence with the set of unordered $n$-tuples of points in $\R P^{n-1}$ in general position (projective n-simplices).  
% \end{prop} 
% \begin{proof} Map a Cartan subgroup to the $n$-tuple of points given by the set of simultaneous eigenvectors. 
 %\end{proof} 

 %Proposition \ref{3 pts} gives a correspondence between limits of points and limits of groups. 
 
 Let ${\mathcal Q}=\{C,F,N_1,N_2,N_3\}$ be the set consisting of the following 5 subgroups of  $SL_3(\R)$ with $a,b >0$, 
$$
\begin{array}{ccccc} 
C & F & N_1 & N_2 & N_3 \\
  \cart , &
\left(
\begin{array}{ccc}
a & t & 0 \\
0 & a & 0 \\
0 & 0 & \frac{1}{a^2} \end{array}
\right) , & 
\left(
\begin{array}{ccc}
1 & s & t \\
0 & 1 & s \\
0 & 0 & 1 \end{array}
\right), & 
 \left(
\begin{array}{ccc}
1 & s & t \\
0 & 1 & 0 \\
0 & 0 & 1 \end{array}
\right), &
\left(
 \begin{array}{ccc}
1 & 0 & t \\
0 & 1 & s \\
0 & 0 & 1 \end{array}
\right).
\end{array} $$ 
  Let $\Gamma$ be the directed graph shown below whose vertices are the elements of $\mathcal{Q}$.

\begin{center}
\begin{tikzpicture}
\node (C) {$C$};
\node[right of=C, node distance=2cm] (F) {$F$};
\node[right of=F, node distance=2cm] (N_1) {$N_1$};
\node[right of=N_1, node distance=2cm](D) {} ;
\node[above of=D, node distance=1cm] (N_2) {$N_2$};
\node[below of=D, node distance=1cm] (N_3) {$N_3$};

\path (C) edge[style={-stealth}] (F);
\path (F) edge[style={-stealth}] (N_1);
\path (N_1) edge[style={-stealth}] (N_2);
\path (N_1) edge[style={-stealth}] (N_3);
\end{tikzpicture}
\end{center}

\begin{theorem}  
\label{maindig}
\begin{enumerate}
\item Every subgroup of $SL_3(\R)$ isomorphic to $\R ^2$ is conjugate to exactly one element of ${\mathcal Q}$. 
\item If $G_1\ne G_2\in {\mathcal Q}$ then $G_2$ is a conjugacy limit of $G_1$ if and only if there is a directed path in $\Gamma$
from $G_1$ to $G_2$.
\item If $G$ is a conjugacy limit of $C$, then $G$ is a 1-parameter limit of $C$. Moreover every directed path in $\Gamma$ is a 1-parameter path of conjugacies. 
 \end{enumerate}

\end{theorem} 
 
  % There are five distinct conjugacy classes of subgroups in $SL_3(\R)$ isomorphic to $\R ^2$:
%$$
%\begin{array}{ccccc} 
%C & F & N_1 & N_2 & N_3 \\
%  \cart , &
%\left(
%\begin{array}{ccc}
%a & t & 0 \\
%0 & a & 0 \\
%0 & 0 & \frac{1}{a^2} \end{array}
%\right) , & 
%\left(
%\begin{array}{ccc}
%1 & s & t \\
%0 & 1 & s \\
%0 & 0 & 1 \end{array}
%\right), & 
% \left(
%\begin{array}{ccc}
%1 & s & t \\
%0 & 1 & 0 \\
%0 & 0 & 1 \end{array}
%\right), &
%\left(
% \begin{array}{ccc}
%1 & 0 & t \\
%0 & 1 & s \\
%0 & 0 & 1 \end{array}
%\right)
%\end{array} $$ 
%where $a, b \in \R_{>0}$ and $s,t \in \R$.  

A {\em configuration} is a finite set $T$ each element of which is a point or projective line in ${\mathbb R}P^2$.
Define ${\mathcal TQ}=\{TC,TF,TN_1,TN_2,TN_3\}$ to be the set of 5 elements which are the configurations shown below.

%\begin{theorem}  
%\label{main}
%\begin{enumerate}

%\item Every subgroup of $SL_3(\R)$ isomorphic to $\R ^2$ is a \emph{1-parameter} limit of $C$.  There is a graph of limits, where an edge denotes that one group is the limit of another. 

%\begin{center}
%\begin{tikzpicture}
%\node (C) {$C$};
%\node[right of=C, node distance=2cm] (F) {$F$};
%\node[right of=F, node distance=2cm] (N_1) {$N_1$};
%\node[right of=N_1, node distance=2cm](D) {} ;
%\node[above of=D, node distance=1cm] (N_2) {$N_2$};
%\node[below of=D, node distance=1cm] (N_3) {$N_3$};

%\path (C) edge[style={-stealth}] (F);
%\path (F) edge[style={-stealth}] (N_1);
%\path (N_1) edge[style={-stealth}] (N_2);
%\path (N_1) edge[style={-stealth}] (N_3);
%\end{tikzpicture}
%\end{center}

%\item 
%There is  a bijection from conjugacy limits of $C$ to {\em characteristic triangle classes}.   
% \end{enumerate} 

%\end{theorem}  

\begin{figure} [H]
 \begin{center}
%\psfrag{0}{$0$}
%\psfrag{1}{$1$}
%\psfrag{2}{$2$}
%\psfrag{3}{$3$}
%\psfrag{4}{$4$)}
%\psfrag{3'}{$3'$}
%\psfrag{4'}{$4'$} 
\includegraphics[scale=.28]{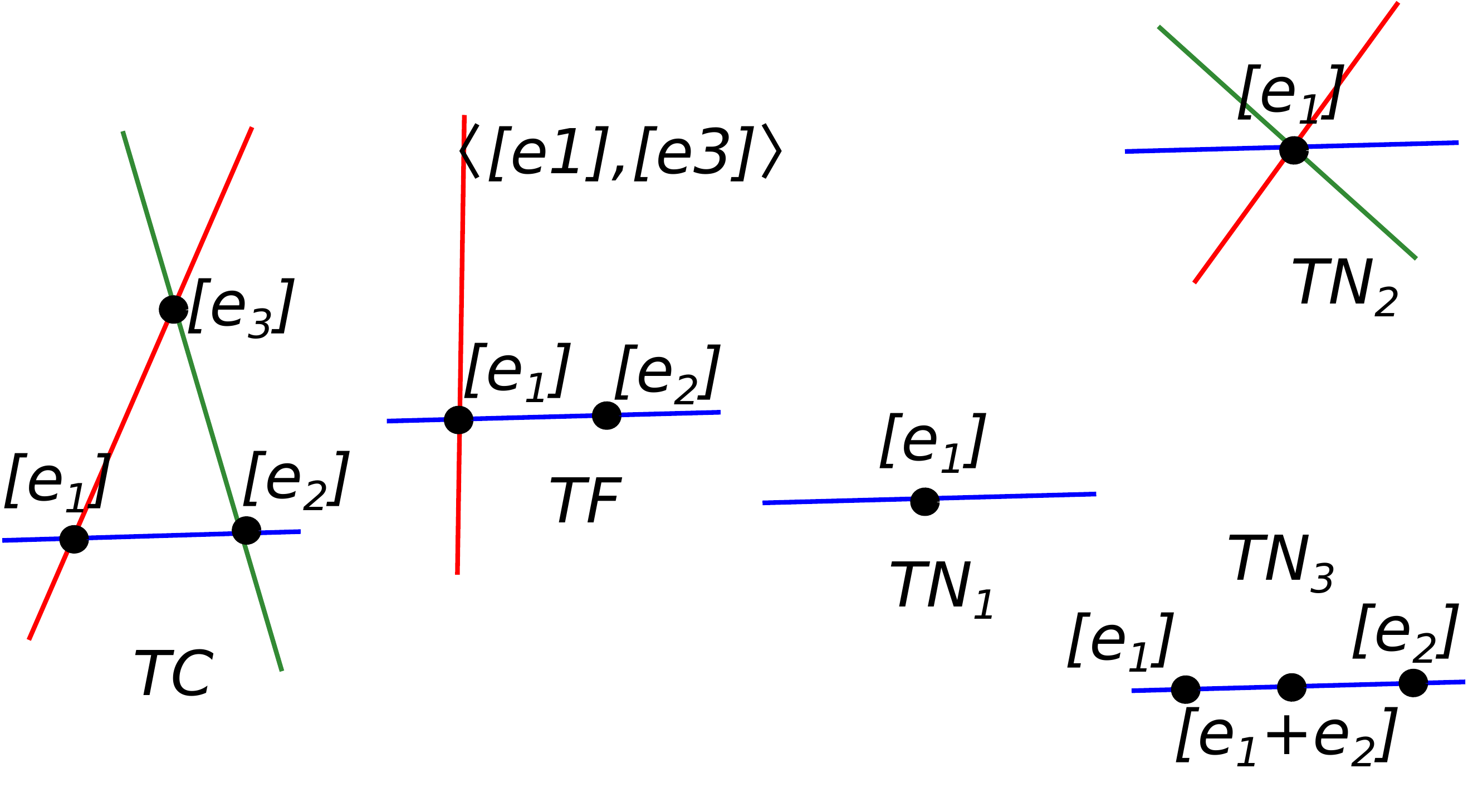}
 \caption{The 5 configurations in $\mathcal{TQ}$} \label{DegenTri}
 \end{center}
 \end{figure}
 
  There is a partial order on ${\mathcal TQ}$ given by inclusion.
 
 \begin{theorem}\label{maintri} There is a bijection $\theta:{\mathcal Q}\longrightarrow {\mathcal TQ}$ defined by
 $\theta(G)=T$ if and only if
\begin{enumerate}
\item $G$ preseverves each element of $T$
\item $T$ is maximal (in the partial order) subject to this condition
\end{enumerate}
\end{theorem}

The {\em hyperreals}  $\HR$ are an ordered field that contains ${\mathbb R}$. (See section \ref{HR}.) A {\em nonstandard triangle}
is a subset $\NS{T} \subset \HR P^2$ consisting of three points $p,q,x \in \HR P^2$ in general position, and the three lines between $p,q$ and $x$.
Let ${\NS G}={\NS G}(p,q,x)\leq SL_3(\HR)$ be the conjugate of the positive diagonal subgroup  which fixes $p,q$, and $x$ .
The {\em shadow} of ${\NS G}$ is the subgroup $sh({\NS G}) \leq SL_3({\mathbb R})$ consisting of those matrices whose entries differ by
an infinitesimal from some matrix in ${\NS G}$.

The next theorem determines exactly when a sequence $P_n \in SL_3(\R)$ has the property $P_n C P_n^{-1}$ converges, and if it does, the conjugacy class of the limit under \emph{any} sequence of matrices.   See section \ref{inf_tri} for the definition of $\alpha$.

\begin{theorem}\label{mainNS} %A nonstandard triangle in $\HR P^2$ determines a limit group.
 % Conjugacy limits of $C$ are in bijection with \emph{equivalence classes of nonstandard triangles}.   Let $\theta$ be the largest infinitesimal angle, and $h$ the longest side:
% Given a nonstandard triangle $\NS T$, let $h$ be the length of the longest side and let $\theta$ be the largest infinitesimal angle in $\NS T$. Set $\theta=1$ 
% if no angle in $\NS T$ is infinitesimal.  Then $sh({\NS G}(\NS T))$
%is conjugate to a group in $\mathcal{Q}$, determined by the following table.
Let $\NS T$ be a nonstandard triangle. Then $sh({\NS G}(p,q,x))$
is conjugate to a group in $\mathcal{Q}$, determined by the following table, where $\alpha$ is a particular hyperreal that depends on $\NS T$.   
$$\begin{array}{|c|c|c||c|}
\hline
 \textrm{ \# infinitesimal angles} & \textrm{\# infinitesimal sides} & \alpha  & \textrm{nonstandard triangle} \\
 \hline
 0& 0 & \textrm{appreciable} & C\\
 1& 1 & \textrm{finite} & F \\
 2 &3 & \textrm{finite} & N_1\\
  \ast & 3 & \textrm{infinitesimal} & N_2 \\
 2 &\ast &\textrm{infinite}& N_3 \\ \hline
\end{array} $$

\end{theorem} 

This result is used to determine %the conjugacy limit of a sequence of conjugates of the
conjugacy limits of the
 positive diagonal group $C \leq SL_3({\mathbb R})$ as follows. The columns of a matrix 
 in $ SL_3({\mathbb R})$ determine a triangle in ${\mathbb R}P^2$. 
 Given a sequence of conjugating matrices $P_n$ we choose a subsequence so
the corresponding triangles $P_n(TC)$ converge to  a nonstandard triangle $\NS T \subset\HR P^2$.
Theorem \ref{mainNS} determines the conjugacy class of the limit $P_n C P_n^{-1}$. %The quantities $h$ and

Most of the work in this paper follows Haettel, \cite{Haettel}, which classifies the homogeneous space of diagonal Cartan subgroup of a group $G$ with the Chabauty topology.   Haettel determines the compactification of the set of all closed abelian subalgebras of dimension the real rank of  $\mathfrak{G}=\mathfrak{sl}_n( \R)$, and shows it is equal to the compactification of the space of limits of Cartan subalgebras for $n=3,4$.  When $n=3$, Haettel shows this space is a simply connected cell complex.  The cells corresponds to the strata in the digraph $\Gamma$. The dimension of the cell decreases as the dimension of the normalizer of the group increases.  To find limits of Lie algebras, Haettel works with sequences under the adjoint action.   
 This paper follows his work, but from the perspective of Lie groups, and introduces the geometric notions of characteristic degenerate triangles, nonstandard triangles, and a maximal configuration preserved by a group.   
 
 Sections \ref{DTC} and \ref{Dig} prove theorems \ref{maintri} and \ref{maindig}, respectively. The second part of the paper is dedicated to the hyperreal theorem \ref{mainNS}.  We first give some background on hyperreal numbers, and explain the lower dimensional case, $SL_2(\R)$.  We build the theory of projective geometry over the hyperreals, to prove theorem \ref{mainNS} in proposition \ref{2to3}. %and show every configuration in $\mathcal{TQ}$ is the shadow of a nonstandard triangle. 
 
 % The first part of the paper is focused on proving theorem \ref{main} from the standard (real) perspective.   The second part of the paper classifies the conjugacy limit of $C$ under \emph{any} sequence of matrices using the hyperreal numbers. 

% Each limit group is a limit in the standard sense under the path given in proposition \ref{1param}. 

\section{Degenerate Triangle Configurations: Proof of Theorem \ref{maintri} }\label{DTC}%%%%%%%%%%%%%%%%%%%Degenerate Triangle Configurations
 
This section proves theorem \ref{maintri}.  We derive the 5 configurations in $\mathcal{TQ}$, and explain how each configuration in $\mathcal{TQ}$ determines a conjugacy limit group in $\mathcal{Q}$. 

Recall a \emph{configuration} is a set with elements which are points and lines in $\R P^2$. 
A configuration, $T$, is a \emph{limit} of a configuration, $S$, if there is a sequence of projective transformations, $P_n$, such that for every $x \in T$, and $n$ sufficiently large, then $x \in P_n S$.  (See \cite{Cooper} definition 2.6). %such that  every element of $T$ is the limit under $P_n$ of an element of $S$, in the Hausdorff topology.  
 We write $P_n S \to T$. A \emph{projective triangle} consists of three points and three line segments connected in the usual way.  A \emph{triangle configuration} in $\R P^2$  is the configuration of three lines in general position, and their three intersection points, obtained by extending the lines of a projective triangle in the natural way.  A \emph{degenerate triangle configuration} is a configuration that is a limit of a triangle configuration. %under a sequence of projective transformations.  
It has at most three points and three lines.   We say two degenerate triangle configurations are in the same equivalence class if they have the same number of points and lines.   %For brevity, we drop the words `equivalence class', and refer to degenerate triangle configurations. 

\begin{figure} [H]
 \begin{center}
%\psfrag{0}{$0$}
%\psfrag{1}{$1$}
%\psfrag{2}{$2$}
%\psfrag{3}{$3$}
%\psfrag{4}{$4$)}
%\psfrag{3'}{$3'$}
%\psfrag{4'}{$4'$} 
\includegraphics[scale=.28]{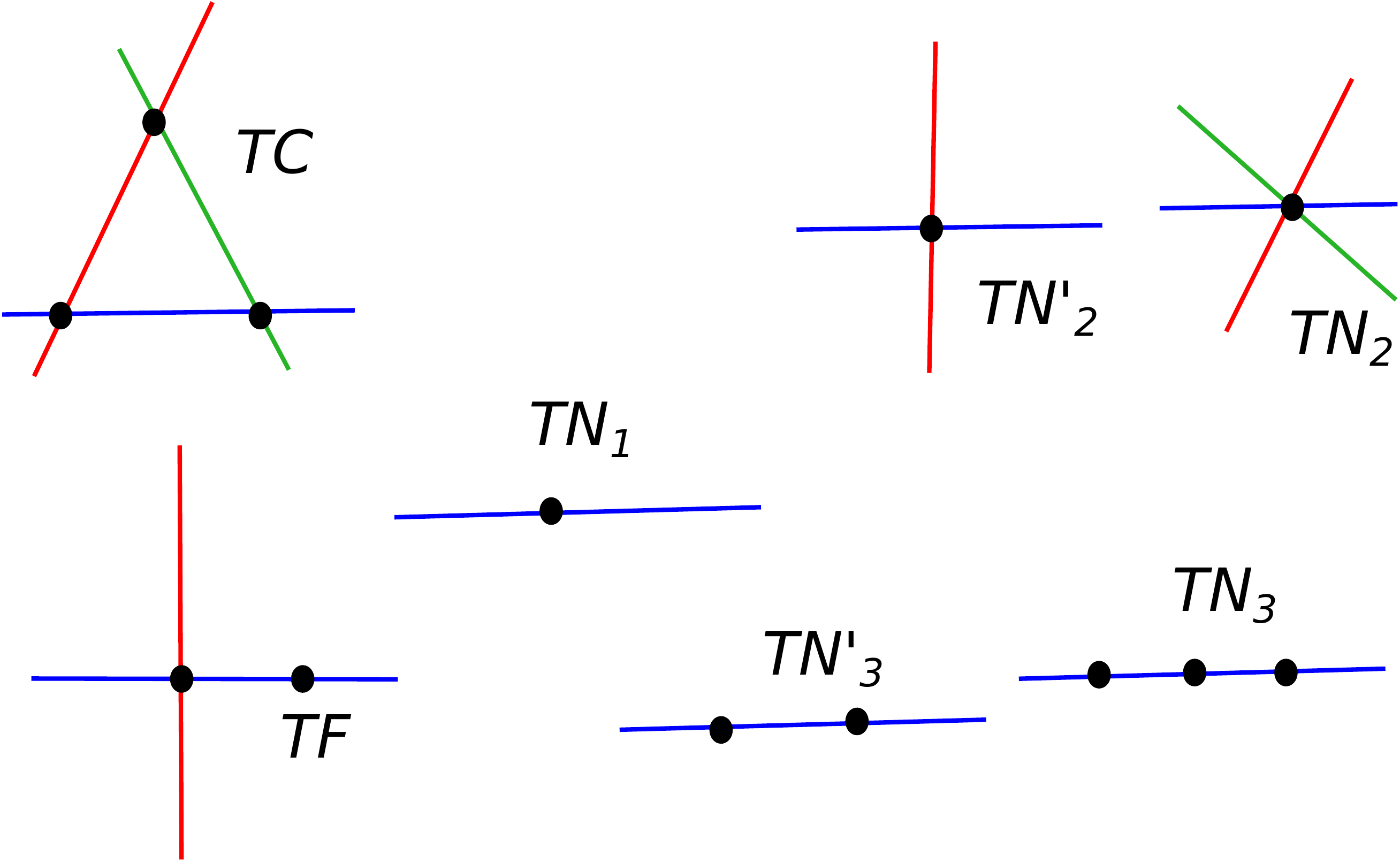}
 \caption{degenerate triangle configurations} \label{7DT}
 \end{center}
 \end{figure}

 Since it is not possible to make a degenerate triangle configuration with 3 lines and 2 points, or 2 lines and 3 points, figure \ref{7DT} shows a representative from every equivalence class of degenerate triangle configurations.  Paths of projective transformations from $TC$ to the degenerate triangle configurations in $\mathcal{TQ}$ are given in proposition \ref{1param}. 
 
 \begin{proposition} \label{pres} Let $C$ be the positive diagonal Cartan subgroup of $SL_3 ( \R)$,  and $S \subset \R P^2$, be the projective triangle configuration preserved by $C$.  Let $P_t \in PSL_3 ( \R)$ be a sequence of projective transformations, such that $S$ converges to  $\displaystyle S_{\infty} = \lim_{t \to \infty} P_t S$, a degenerate triangle configuration, and $G$ has conjugacy limit $\displaystyle G_{\infty} = \lim_{t \to \infty} P_t G P_t ^{-1}$. Then $G_{\infty}$ preserves $S_{\infty}$. 
\end{proposition}

\begin{proof} Let $G_t := P_t C P_t^{-1}$, so that $G_t$ preserves $S_t:= P_t S$ for all $t$.  Suppose for contradiction that $G_{\infty}$ does not preserve $S_{\infty}$.  Then there is some $ x \in S_{\infty}$, and $g \in G_{\infty}$ such that $g x  \not \in S_{\infty}$, or $d( g x , S_{\infty}) > 0$.  Take a sequence $g_t \in G_t$ such that $\displaystyle \lim_{t \to \infty } g_t = g$. Pick a point $x_0 \in S$ so that $\displaystyle \lim_{t \to \infty } P_t x_0 = x$, so $P_t x_0 \in S_t$. %thus $y \in S_t$ for some $t$.  
But then $\displaystyle \lim _{ t \to \infty } d ( g_t (P_t x_0)  , S_t ) = d ( g x  , S _ { \infty} ) > 0$, contrary to $d ( g_t (P_t x_0) , S_t ) =0$.  \qed
\end{proof}

The set of degenerate triangle configurations is partially ordered by inclusion.    Given a group, $G \leq SL_3(\R)$, a  degenerate triangle configuration, $T$, is \emph{characteristic} for $G$, if $G$ maps each point or line of $T$ to itself, and $T$ is maximal subject to the partial order.   %The characteristic degenerate triangle configurations in figure \ref{DegenTri} are the limits of a projective triangle under the sequences of transformations in proposition \ref{1param}.  
 There are degenerate triangle configurations which are not characteristic for any group in $\mathcal{Q}$.  They are shown in figure \ref{7DT}: $TN'_2$ and $TN'_3$.

   %For example, the configuration of one point at the intersection of two lines. 
%Recall that a group action on a projective line has either two fixed points (hyperbolic), one fixed point (parabolic), or fixes every point (identity).  Any group which preserves the configuration of one point at the intersection of two lines also preserves more lines through the point (if the group action on both lines is parabolic), or another point on one of the lines (if the group action on a line is hyperbolic).  Therefore this configuration is not characteristic for any group. %Similarly, there is a degenerate triangle configuration which consists of one line with two points.  
%Similarly any group which preserves the configuration of two points on a line also preserves more points on the line (if the group acts trivially on the line), or an additional line through one of the points (if the group action on the line is hyperbolic).  

\begin{proof}\textit{of Theorem \ref{maintri}:} %Let $G \in \mathcal{Q}$ be a conjugacy limit group.  Then $\theta(G) =T \in \mathcal{TQ}$ is a characteristic triangle for $G$, so conditions 1 and 2 are satisfied.  
The notation has been chosen so that for every $G \in \mathcal{Q}$, it is easy to check $\theta(G) =TG \in \mathcal{TQ}$ is a characteristic degenerate triangle configuration for $G$.    That is, $G$ preserves every point or line of $\theta(G)$, and $\theta(G)$ is maximal in the partial order.
\qed
\end{proof} 

 %We call an equivalence class of characteristic degenerate triangle configurations a \emph{characteristic triangle class}.  Figure \ref{DegenTri} shows the 5 characteristic triangle classes.
Given a group $G \leq SL_3(\R)$, a \emph{characteristic triangle class}, $[T]$ for $G$, is an equivalence class of degenerate triangle configurations such that each $S \in [T]$ is characteristic for a group conjugate to $G$. Figure \ref{DegenTri} shows the set $\mathcal{TQ}$, which consists of one representative of each characteristic triangle class.  %for every conjugacy limit group in $\mathcal{Q}$.   %The elements of $\mathcal{TQ}$ (figure \ref{DegenTri}) are limits of $TC$ under the paths given for the conjugacy limit groups in proposition \ref{1param}.
By theorem \ref{maindig}, $\mathcal{Q}$ contains one representative of each conjugacy class of conjugacy limit group of $C$. 
Therefore, $\theta$ induces a bijection from conjugacy classes of conjugacy limit groups of $C$ to characteristic triangle classes.  This map is well defined: If $G_1$ and $G_2$ are conjugacy limits of $C$, and $T_1, T_2$ are their respective characteristic degenerate triangle configurations, and $P \in GL_3(\R)$, then $PG_1 P^{-1}= G_2$ if and only if $PT_1 = T_2$.

Given a group $G \leq SL_3(\R)$, the \emph{maximal configuration} preserved by $G$ is the set whose elements are the points fixed by $G$ and the lines preserved by $G$.  It might be infinite.  For example, the maximal configuration for $N_2$ consists of every line through a point, and the maximal configuration for $N_3$ is every point on a line. 

% We say a configuration, $T$, is \emph{maximal} for a group, $G \leq SL_3(\R)$,  if $G$ acts on $\R P ^2$ so that each element of $T$ is mapped to itself, and if $G$ preserves a point or line, then this point or line is in $T$. Notice that a maximal configuration may contain infinitely many points or lines.  

In configuration $TN_3$, three points on the line are fixed, and any group which acts on $\R P^2$ and fixes three points on a projective line must fix every point on the line.  In configuration $TN_2$, three lines in the link of the vertex are preserved, and similarly, every line in the link is preserved.  %Under this interpretation, figure \ref{DegenTri} shows the %triangle configuration which is maximal 
%maximal configuration
%for each group.  
In this way, each configuration in figure \ref{DegenTri} determines a maximal configuration. These ideas will be investigated further with the hyperreal viewpoint.

\section{The Digraph of Conjugacy Limit Groups: Proof of Theorem \ref{maindig}}\label{Dig}%%%%%%%%%%%%%%%%%%%%%%%

In this section we prove theorem \ref{maindig}. %We say two representations $\rho_1 , \rho_2 \in \textrm{Hom}( \R^2, SL_3 (\R))$ are equivalent if $\rho_1 = \tau \rho_2 \phi$, where $\phi \in \textrm{Aut} (\R^2)$ and $\tau \in \textrm{Inn} (SL_3(\R))$. 
Part 1 follows easily from work of Haettel \cite{Haettel}, or the classification over $\mathbb{C}$ of Suprenko and Tyshkevitch, \cite{SupTysh}, p.134.   Next we prove part 3: every element of $\mathcal{Q}$ is a conjugacy limit of $C$. 
If each element of the sequence of conjugating matrices lies in a one parameter subgroup of $SL_n(\R)$, then $G$ is a conjugacy limit \emph{under a 1-parameter path of conjugacies}.  
%We give a 1-parameter path to each conjugacy limit in part 2 of theorem \ref{main}. 

\begin{proposition}\label{1param}
Each conjugacy limit group is a limit under a 1-parameter path of conjugacies. 
\end{proposition}
\begin{proof}  Each conjugacy limit group is a limit under the 1 parameter path of conjugacies (as $n \to \infty $): 
$$
\begin{array}{cccc} 
C \to F&\qquad C \to  N_1 & \qquad C \to N_2 & \qquad C \to N_3 \\
\left(
\begin{array}{ccc}
1 & n & 0 \\
0 & 1 & 0 \\
0 & 0 & 1 \end{array}
\right) , 
& \qquad  \left(
\begin{array}{ccc}
1 & n & \frac{n^2}{2} \\
0 & 1 & n \\
0 & 0 & 1 \end{array}
\right),  
& \qquad  \left(
\begin{array}{ccc}
1 & n & n \\
0 & 1 & 0 \\
0 & 0 & 1 \end{array}
\right),
& \qquad \left(
 \begin{array}{ccc}
1 & 0 & n \\
0 & 1 & n \\
0 & 0 & 1 \end{array}
\right) . 
\end{array}  $$   \qed \end{proof}

To finish the proof of theorem \ref{maindig}. We must show:\\
 (1) there is a 1-parameter path of conjugacies for every arrow in the digraph, and \\
 (2) no other arrows are possible. \\
  The path $C \to F$ is given in proposition \ref{1param}.  For the rest of the arrows, use the paths: 
 %$$  
% \begin{array}{ccc} 
% F \to N_1  \qquad &   N_1 \to N_2  \qquad &  N_1 \to N_3 \\
%\left( \begin{array}{ccc} 
% n &0 & n^2\\
% 0&1&n\\
% 0&0&1 
% \end{array} \right)
% \qquad 
%&
% \left ( \begin{array}{ccc} 
% n&0 & 0\\
% 0 &1 &0 \\
% 0&0&n \end{array} \right)
% \qquad
%&
% \left( \begin{array}{ccc} 
% \frac{1}{n^2} &n & 0 \\
% 0 & \frac{1}{n} & n^2 \\
% 0 & 0 & \frac{1}{n^2} \end{array} \right).
% \end{array} 
% $$
  $$  
 \begin{array}{ccc} 
 F \to N_1  \qquad &   N_1 \to N_2  \qquad &  N_1 \to N_3 \\
\left( \begin{array}{ccc} 
 1 &n & \frac{n^2}{2}\\
 0&1&n\\
 0&0&1 
 \end{array} \right)
 \qquad 
&
 \left ( \begin{array}{ccc} 
 1&0 & 0\\
 0 &1 &0 \\
 0&0&\frac{1}{n} \end{array} \right)
 \qquad
&
 \left( \begin{array}{ccc} 
 \frac{1}{n} &0 & 0 \\
 0 & 1 & 0 \\
 0 & 0 & 1 \end{array} \right).
 \end{array} 
 $$
as $n \to \infty$. This shows every directed path in $\Gamma$ is realized by a 1-parameter limit.  To show these are all possible arrows, we use:

\begin{proposition}\label{normalizer}[\cite{Cooper}, proposition 3.2]  If $A,B \leq G=SL_n (\R)$ and $A$ is a conjugacy limit of $B$, then the dimension of the normalizer must increase: $\dim N_ G (A) \leq \dim N_G(B)$, with equality if and only if $A$ and $B$ are conjugate. 
\end{proposition} 

As a corollary of proposition \ref{normalizer}, we see if we have both conjugacy limits $A \to B$ and $B \to A$, then $A$ and $B$ are conjugate. Computing the dimension of the normalizer shows there are no arrows going backwards in the digraph: $\dim N_G (C) = 2, \dim N_G (F) = 3, \dim N_G(N_1) = 4, $ and $\dim N_G(N_2) = \dim N_G (N_3) = 5$, where $G= SL_3(\R)$.

It remains to show neither $N_2$ or $N_3$ is a limit of the other.  %Since $N_2$ and $N_3$ are dual, if $N_2 \to N_3$ then $N_3 \to N_2$ (and vice versa). 
Since $\dim N_G(N_2) = \dim N_G (N_3)$, if $N_2 \to N_3$, then  proposition \ref{normalizer} would imply $N_2$ and $N_3$ are conjugate, which is false.  Therefore, neither $N_2$ or $N_3$ can limit to the other.   This finishes the proof of theorem \ref{maindig}. \qed

\begin{remark}
Notice $N_3$ fixes every point on a line, and $N_2$ preserves every line through a fixed point.  The maximal configurations for $N_2$ and $N_3$ are dual.  Moreover, duality induces an automorphism of the digraph of conjugacy limits. 
\end{remark}

%Again, we interpret this limit on characteristic degenerate triangle configurations as a rearrangement of the vertices of a nonstandard triangle in $\HR P^2$.  
%By theorem \ref{maintri}, conjugacy limits groups of $C$ in $\mathcal{Q}$ are in bijection with characteristic degenerate triangle configurations in $\mathcal{TQ}$. %By proposition \ref{pres}, if $A$ is a conjugacy limit group of $C$ by a sequence of matrices $P_n$, then $A$ must preserve the degenerate triangle configuration which is a limit of $TC$ under $P_n$.  
%Therefore $\Gamma$ may be interpreted as a graph of limits of degenerate triangle configurations. 

%  Proposition \ref{normalizer} implies that if $A \to B$ the dimension of the orbit closures of $B$ is smaller than the dimension of the orbit closures of $A$, since
%$$ \dim N_G(A) + \dim (G/ N_G(A)) = \dim G,  $$
 %So the maximal configuration preserved by $B$ may be a proper subset of the maximal configuration preserved by $A$. 
%and $\dim (G/ N_G(A))= \dim (\textrm{orbit closures of } A)$.  
Suppose $H,L \leq SL_n(\R)$ are not conjugate, and $H \to L$. We say the action from $H$ to $L$ is \emph{decompressed} if there is an orbit closure of $H$ which is a union of infinitely many orbit closures of $L$.  In the conjugacy limit  $N_1 \to N_2$, the action on the link of the fixed point is decompressed, so every line through the fixed point is preserved.  This is the maximal configuration for $N_2$: a single point is fixed and every line through the point is preserved.  In the conjugacy limit $N_1 \to N_3$, the action on the preserved line is decompressed, so every point on the line is fixed.  This is the maximal configuration for $N_3$: a line with every point on the line fixed. 

We concluded this section with a final property satisfied by real conjugacy limits, which we will use in corollary \ref{coords}.  %Every conjugacy limit group under an arbitrary sequence of matrices is conjugate to a conjugacy limit group by an upper triangular sequence of matrices.   

 \begin{lemma} \label{uppertri}
  Suppose $H$ is a subgroup of $SL_3(\R)$, and $P_n$ is a sequence of conjugating matrices such that $P_n H P_n ^{-1}$ has conjugacy limit $L \leq SL_3 ( \R)$.  There is a sequence of upper triangular matrices, $P_n '$ such that $P_n' H {P_n'} ^{-1}$ converges to a conjugate of $L$. 
 \end{lemma}
 \begin{proof} Recall the Iwasawa decomposition of a matrix, $P =KNA$, where $K$ is orthogonal, $N$ is unipotent, and $A$ is diagonal.  Writing each $P_n$ in this way, we have $ P_n H P_n ^{-1} = K_n (N_n A_n H A_n ^{-1} N_n ^{-1}) K_n ^{-1}$.   The orthogonal group is compact, so every sequence has a convergent subsequence, and in particular, every sequence $K_n H' K_n ^{-1}$ converges to a conjugate of $H'$.  Thus we may assume $P_n= N_nA_n$, so $P_n$ is upper triangular.  \qed
 \end{proof}

 \section{The Hyperreals, $\HR$}\label{HR}%%%%%%%%%%%%%%%%%%%%%%%%%%

The second part of the paper is devoted to proving theorem \ref{mainNS}.   For a good introduction to the hyperreals, see \cite{Gold}.  %We will work with the hyperreal numbers, $\HR$. 
The hyperreal numbers, $\HR$, are a non-Archimedian field, with $\R \subset \HR$. % Elements of $\HR$ are equivalence classes of sequences of real numbers.   

Fix $\mathcal{F}$, a non principal ultra filter on $\N$.  Let $\alpha =( a_n)$ and $\beta = (b_n)$ be sequences of real numbers.   Define an equivalence relation $\sim$ on sequences of real numbers by $(a_n) \sim (b_n)$ if $\{ n \in \N : b_n = a_n \} \in \mathcal{F}$. % If $\sim$ holds, we say that as hyperreal numbers $\alpha = \beta $.   
A hyperreal number is an equivalence class $[(\alpha_n)]$. 
There is a natural inclusion $\R \hookrightarrow \HR$ which sends a real number to a constant sequence. Hyperreal addition and multiplication are defined component-wise.   The hyperreals are ordered by $\alpha \leq \beta $ if $\{ n \in \N : a_n \leq b_n \} \in \mathcal{F}$. The absolute value is defined in any ordered field.  

\begin{definition} \label{hyperreal}
\begin{enumerate} 
\item A hyperreal number $\alpha$ is \emph{infinite} if $|\alpha |> n$ for every $n \in \N$.  A hyperreal number $\alpha $ is \emph{infinitesimal} if $|\alpha| < \frac{1}{n}$ for every $n \in \N$.  If $\alpha$ not infinite, $\alpha$ is \emph{finite}.  Denote the set of finite hyperreals by $\mathbb{F}$.  
%\item Any hyperreal $\alpha = r + \beta$ where $r$ is appreciable and $\beta$  is infinite. The finite part of $\alpha$ is $\textrm{Fin}(\alpha) = r$. 
\item Let $\alpha \in \HR$ be a finite hyperreal.  The \emph{standard part} or \emph{shadow} of $\alpha$  is $sh(\alpha) \in \R$, where $ \alpha - sh(\alpha)$ is infinitesimal.  Note that $sh: \mathbb{F} \to \R$ is a ring homomorphism.
\item A hyperreal $\alpha \in \HR$ is \emph{appreciable} if $\alpha$ is neither infinitesimal or infinite,  i.e., $\alpha $ is a non-zero real number plus an infinitesimal. 
 \item The \emph{galaxy} of $ \alpha \in \HR$ is $\textrm{Gal}(\alpha) = \{ x \in \HR : \alpha -x \in \mathbb{F} \}$.
 \item The \emph{$\varepsilon$-galaxy} of $\alpha$ in $\HR$ is $\textrm{Gal}_{\varepsilon}(\alpha) = \{ x \in \HR : \alpha -x  \in \varepsilon \cdot \mathbb{F} \}$.

 \end{enumerate}
\end{definition}

% One can check (or see \cite{Gold}) that 
Note a finite hyperreal may be either appreciable or infinitesimal. Clearly $\textrm{Gal}$ and $\textrm{Gal}_\varepsilon$ define equivalence relations on $\HR$.  We denote hyperreal objects in script $\NS{G}$ or $\NS{L}$, and denote their standardizations $G$ or $L$. We use the usual inner product, $\langle x,y \rangle = x \cdot y$ for $x,y \in \R ^n$ or $\HR ^n$.

\begin{definition}
\begin{enumerate} 
 \item A \emph{projective basis} for $\R P^n$ (or $\HR P^n$) consists of $n+2$ equivalence classes of vectors,  such that any $n+1$ vectors form a basis for the underlying vector space.  %We will write \emph{basis} when we mean projective basis. %and use the symbol $\mathcal{B}$ for both.  The meaning should be clear from the context. 
 The word basis means either vector space basis or projective basis, depending on the context. 
  \item The \emph{usual basis} for $\HR^{n+1}$ (or $ \R ^{n+1}$) is $\{e_0,e_1, ... ,e_n\}$. The \emph{usual basis} for $\HR P^n$ (or $\R P^n$) is $\{ [e_0],... ,[e_n], [e_0 + e_1+ \cdot \cdot \cdot +e_n] \}$. 
  \item  The \emph{shadow} map is $sh: \HR P^n \to \R P^n$ where $[v] \mapsto [sh(\frac{v}{||v||})] $, and we take the shadows of the coordinates.  The \emph{shadow} of a basis $\NS{B}\subset \HR P^n $  is $sh( \NS{B}) =\{ [sh(v)] | v \in \NS{B}\} \subset \HR P^n$.
    \item  A hyperreal projective basis $\NS{B}$ is \emph{appreciable} if $sh(\NS{B})$ is a projective basis for $\R P^n$.  A hyperreal projective transformation is \emph{appreciable} if the image of some appreciable basis is an appreciable basis.

 \end{enumerate} 
 \end{definition} 
 
 % The shadow of a nonstandard basis is well defined, since $sh(-v) = - sh(v)$. 
 
 Notice the shadow of a hyperreal basis may not be a real basis, since shadows of basis  elements  could be the same!   A hyperreal transformation $[A] \in PGL_n(\HR)$ is finite if and only if there exists $[B] \in PGL_n(\R)$ and $\lambda \in \HR$ such that $B- \lambda A$ is infinitesimal.   Every finite hyperreal projective transformation differs from a real projective transformation infinitesimally. 

%The hyperreal numbers may be endowed with the interval topology, where any open interval $(a, b)$ with $a, b \in \HR$ may be a neighborhood, and a set is open if it is a union of intervals.  (See \cite{Gold} for more on the interval topology and other possible topologies on $\HR$.)   Then $GL_n(\HR)$ inherits a topology as an open subset of the real vector space $M_n(\HR)$, just as $GL_n (\R)$ does from $M_n(\R)$. 
 
\begin{definition}Given $\NS{G}  \leq SL_n ( \HR)$, the \emph{finite part}, $\textrm{Fin}(\NS{G})$, is the subset of all elements that have finite entries.  The \emph{subset of infinitesimal elements}, $\NS{I} $, is 
%the infinitesimal neighborhood of the identity.  $\NS{I}$ is 
the set of matrices that are the identity matrix plus a matrix with infinitesimal entries. 
\end{definition}

\begin{lemma} 
 $\textrm{Fin}(\NS{G})$ and $\NS{I}$ are subgroups of $\NS{G}$. 
 %both independent of choice of basis. 

\end{lemma}
\begin{proof} 
Sums and products of finite hyperreals are finite, so $\textrm{Fin}(\NS{G})$ is closed under multiplication. Let $A \in \textrm{Fin}(\NS{G}) \leq SL_{n}( \HR)$. Since sums and products of finite hyperreal numbers are finite, $\textrm{Adj} A$ is finite. Since $\textrm{det}A= 1$, then $A ^{-1} = \frac{\textrm{Adj} A}{\det A} \in \textrm{Fin}(\NS{G})$.   Thus $\textrm{Fin}(\NS{G})$ is a group.  Similarly, $\NS{I}$ is a group.  \qed
%Suppose $\NS{B '}$ is another appreciable basis. Then $\NS{B'}$  is the image of $\NS{B}$ under an appreciable projective transformation, so the entries of $\textrm{Fin}(\NS{G})$ are finite in $\NS{B '}$.  The proof that $\NS{I}$ is a subgroup independent of choice of basis is analogous. 
\end{proof} 

\begin{definition} Given $\NS{G} \leq SL_n (\HR)$ and $\NS{A} \in \textrm{Fin}(\NS{G})$, the \emph{shadow}, $sh(\NS{A})$, of $\NS{A}$,  has entries that are the shadows of entries of $\NS{A}$.  The \emph{standard part} or \emph{shadow} of $\NS{G}$ is $\textrm{sh}(\NS{G}) := \{ sh(\NS{A}) | \NS{A} \in \textrm{Fin}(\NS{G}) \}$.
\end{definition} 

 \begin{lemma}  $sh( \textrm{Fin} (\NS{G})) \cong \textrm{Fin}(\NS{G})/ (\NS{I}\cap \textrm{Fin}(\NS G)) $.
 \end{lemma} 
 \begin{proof} Note  $sh : \textrm{Fin}(\NS{G}) \to sh (\textrm{Fin}(\NS{G}))$ is a homomorphism since %Let $\NS A, \NS B \in \textrm{Fin}(\NS G)$. Since the map is defined by taking the shadow of each entry then, 
 the map is defined by taking the shadow of each entry, and $sh: \mathbb{F} \to \R$ is a ring homomorphism. 
 %$$sh( \NS A \NS B) = [ sh(\NS A \NS B )_{ij}] = [sh (\NS A) _{ij}] [ sh( \NS B) _{ij}] = sh \NS A sh \NS B. $$
  The kernel is $\NS{I} \cap \textrm{Fin}(\NS G)$.  The map $sh$ is surjective since $\R \hookrightarrow \HR$.  Apply the first isomorphism theorem. \qed
%Suppose $\NS{B}'$ is another appreciable basis, then $\NS{B}'$ is the image of $\NS{B}$ under an appreciable projective transformation, $\NS{\theta} \in \textrm{Fin}(\NS{G})$.  Since $[\NS{A}]_{\NS{B}'} = \theta \NS{A} \theta^{-1}$, we see $sh(\NS{A})$ is independent of choice of basis. \\
 % Then $| \theta \NS{A} \theta^{-1} -  sh(\theta) \NS{A} sh(\theta^{-1}) |$ is infinitesimal.  Since two numbers that are infinitesimally close have the same shadow,  $sh( \theta \NS{A} \theta^{-1})=sh( sh(\theta) \NS{A} sh(\theta^{-1}))= sh(\theta) sh( \NS{A}) sh( \theta^{-1}), $ where the second equality holds since $sh( \theta)$ is a real transformation.  So, $sh([\NS{A}]_{\NS{B}})$ and $sh([\NS{A}]_{\NS{B}'})$ are conjugate by a real projective transformation. 
 \end{proof} 

%Let $\R\textrm{-Vect}$ denote the category of  real vector spaces and invertible linear transformations, and $\HR  \textrm{-Vect}$ denote the category of hyperreal vector spaces and invertible hyperreal linear transformations.  Then  $ \otimes_{\R} \HR: \R \textrm{-Vect} \hookrightarrow \HR \textrm{-Vect}$, where $ \otimes_{\R} \HR(V) = V \otimes_{\R} \HR$, is a faithful functor.   In general,  $\otimes_{\R} \HR$ is not surjective, since $\HR  \textrm{-Vect}$ contains infinite hyperreal transformations, but the image of $\otimes_{\R} \HR$ contains only finite hyprereal transformations.  Let $\mathbb{P}( \otimes_{\R} \HR): \R \textrm{-Proj} \hookrightarrow \HR \textrm{-Proj}$ be the induced functor on the categories of projective space and projective transformations. Then $\mathbb{P}( \otimes_{\R} \HR)$ is surjective on objects.  Given a morphism in $\HR \textrm{-Proj}$, we may scale it by a diagonal matrix to be finite.  Since any hyperreal projective transformation differs from a transformation in the image of $\mathbb{P}( \otimes_{\R} \HR)$ by an infinitesimal transformation, the image of the set of real projective transformations under $\mathbb{P}( \otimes_{\R} \HR)$ is dense in the set of hyperreal projective transformations.  Therefore $\mathbb{P}( \otimes_{\R} \HR)$ is essentially surjective. 

Instead of writing this discussion in terms of matrices, we could have considered nonstandard projective transformations in the context of \emph{appreciable} (vector space or projective) bases.  %Suppose a projective transformation maps an appreciable basis $\NS{B}$ to an appreciable basis.  
%Let $\NS{B}$ be an appreciable vector space basis. If $\NS{B}'$ is another appreciable basis, then $\NS{B}'$ is the image of $\NS{B}$ under an appreciable linear transformation, so the entries of $\textrm{Fin}(\NS{G})$ are finite in $\NS{B '}$. Since any subset of $n+1$ elements of a projective basis form a basis for the underlying vector space, we see an appreciable projective transformation maps any appreciable projective basis to another appreciable projective basis. 
%The other proofs are similar. 
It is easy to show a hyperreal projective transformation is appreciable if and only if the image of every appreciable basis is an appreciable basis. 

 A \emph{nonstandard triangle configuration} consists of three lines in general position in $\HR P^2$, which intersect in three distinct points.  Let $\NS{P}$ be a matrix of hyperreals, and $\NS{T}$ be the nonstandard triangle configuration which is the image of a projective triangle configuration under $\NS{P}$. Let  $\NS{C} \leq SL_3(\HR)$ be the group of positive diagonal (hyperreal) matrices.  A conjugate of $\NS{C}$ by  $\NS{P}$ is uniquely determined as the stabilizer of the vertices of $\NS{T}$.  %Every nonstandard triangle determines a \emph{nonstandard triangle configuration}, by extending the lines in the obvious way, which we will also denote $\NS{T}$. 
 
% \begin{lemma}\label{shT}
% The shadow of a nonstandard triangle configuration is the degenerate triangle configuration that is the (Hausdorff) limit of images of a projective triangle configuration under a sequence of projective transformations.
% \end{lemma}

%We claim that the shadow map is well defined and a bijection from equivalence classes of nonstandard triangles to characteristic triangle classes.   First we show:

%Taking the shadow of nonstandard triangle configurations determines an \emph{equivalence relation on the set of degenerate triangle configurations}. Therefore,
%Recall the 7 degenerate triangle configurations are divided into 5 \emph{characteristic triangle classes}.  %A characteristic triangle class consists of the set of images of all nonstandard triangles in an equivalence class of nonstandard triangle configuration of theorem \ref{mainNS}. %The configuration of one point at the intersection of two lines, and one point at the intersection of three lines are in the same equivalence class; and the configurations of two points on a line, and three points on a line belong to the same equivalence class.    
%The configurations $N_2$ and $N'_2$ belong to the same equivalence class, as do $N_3$ and $N'_3$. 
%We take a characteristic degenerate triangle configuration as a representative of each characteristic triangle class, as shown in figure \ref{DegenTri}. 
   
    \begin{theorem}\label{HRlimgp} 
    Suppose $P_n \in GL_3(\R)$ and $P_n C P_n^{-1}$ converges to $L$.  Define $\NS P: = [P_n] \in GL_3 (\HR)$, and $\NS G := \NS P \NS C \NS P^{-1}$.  Then
    \begin{enumerate} 
    \item $sh (\textrm{Fin}(\NS G)) = L$
    \item $\NS G$ is the conjugate of $\NS C$ that preserves the points $\NS P ([e_1], [e_2], [e_3]) \subset \HR P^2$. 
   \item the nonstandard triangle configuration $\NS T$ with vertices $\NS P ([e_1], [e_2], [e_3])$ has shadow $T = sh (\NS T)$ a degenerate triangle, and $L$ preserves $T$. 
 % \item   Every conjugacy limit group of $C \leq SL_3(\R)$ is given by $sh(\textrm{Fin}(\NS{G}))$, where $\NS{G}$ is a conjugate of $\NS{C} \leq SL_3(\HR)$ by a hyperreal matrix. 
%  \item Every degenerate triangle configuration, $T$, is the shadow of a nonstandard triangle configuration, $\NS{T}$.  
%  \item The group $\NS{G}$ preserves $\NS{T}$, and $sh(\textrm{Fin}(\NS{G}))$ preserves $T$. 
    \end{enumerate} 
 \end{theorem}

\begin{proof} 
% A hyperreal number is an equivalence class of a sequence of real numbers, so a  sequence of conjugating matrices, $P_n$  determines a hyperreal conjugating matrix, $\NS{P}$. The sequence $P_n$ may also be viewed as a sequence of projective transformations, which determines a hyperreal projective transformation $\NS{P}$. 
 %Conjugating the diagonal Cartan subgroup in $SL_3 ( \HR)$ by a matrix of hyperreal numbers gives a subgroup $ (\HR ) ^ 2 \cong \NS{G} \leq SL_3( \HR)$.  The standard part of $\textrm{Fin}(\NS{ G})$ is a limiting group.  
 
(1) Conjugating $\NS{C}$ by $\NS{P}$,  gives $ (\HR ) ^ 2 \cong \NS{G} \leq SL_3( \HR)$. %the group which preserves the nonstandard triangle configuration, $\NS{T}$,  the image of a projective triangle configuration under $\NS{P}$, by lemma \ref{shT}. %Let $T := sh (\NS{T})$ be the corresponding characteristic representative for this characteristic triangle class as in lemma \ref{NSeqrel}.  Then  $G(T) := sh(\textrm{Fin}(\NS{G}(\NS{T})))    \leq SL_3(\R)$, the group which preserves the characteristic triangle for $T$, see figure \ref{DegenTri}.  For dimension reasons, (\cite{Cooper} Proposition 3.1), $G(T) \cong \R ^2$, and $G(T)$ is a conjugacy limit under the equivalence class of sequence of matrices given by $\NS{P}$. 
 For dimension reasons, (see \cite{Cooper} Proposition 3.1), $sh(\textrm{Fin}(\NS{G}))\cong \R ^2$.  Every subgroup of $SL_3(\R)$ isomorphic to $\R^2$ is a conjugacy limit of $C$, so $sh(\textrm{Fin}(\NS{G}))$ is a conjugacy limit of $C$.  By assumption, $[P_n] =\NS P$, therefore $sh(\textrm{Fin}(\NS{G}))=L$.
  
 %(2) %Recall $sh(\NS{T})$ is the shadow of the image of a projective triangle configuration under a nonstandard projective transformation, which is an equivalence class of a sequence of real projective transformations. 
% A nonstandard triangle is the image of a projective triangle under $\NS{P}$.  %Let $P_n$ be a sequence of (real) projective transformations in the equivalence class of $\NS{P}$.  
% The degenerate triangle that is the image of a projective triangle under $P_n$ is the shadow of the image of a projective triangle under $\NS{P}$.

(2)   Since $\NS C$ preserves $([e_1],[e_2] ,[e_3])$, then $\NS G= \NS{ P C P}^{-1}$ preserves $\NS P ([e_1], [e_2], [e_3])$. 
 
 (3) This follows from proposition \ref{pres}.  \qed
 \end{proof} 

%\begin{rmk}\label{corrsp HR}  Each entry in a hyperreal matrix is given as a sequence $(p_n)$.   So  a sequence of conjugating matrices $(P_n ) _ { n \in \N} \in SL_m ( \R)$ is a single hyperreal matrix in $\NS{P} \in SL_m ( \HR)$.  
%\end{rmk}  

 \section{Conjugacy Limits of the Diagonal Cartan subgroup in $SL_2 (\R)$}%%%%%%%%%%%%%%%%%%%%%%%%%%%%%%%%%%%%
 
 In this section, we show how to view a limit of the diagonal group in $SL_2(\R)$ as the shadow of the finite part of a conjugate of the diagonal group in $SL_2 (\HR)$. 
 
\begin{theorem} The Cartan subgroup $ \{  \big(\begin{smallmatrix}
a&0\\ 0& \frac{1}{a}
\end{smallmatrix} \big) | a >0 \} \leq SL_2 (\R)$,  has two conjugacy limits: the Cartan subgroup and the parabolic group
$ \{  \big(\begin{smallmatrix}
1&t\\ 0& 1
\end{smallmatrix} \big)|  t \in \R\} $.  The Cartan subgroup preserves the maximal configuration consisting of two fixed points on a projective line, and the parabolic group preserves the maximal configuration consisting of a projective line with one fixed point. 
\end{theorem}
\begin{proof} 
There are three conjugacy classes of 1-parameter subgroups in $SL_2(\R)$  : elliptic, parabolic, and hyperbolic.  (See  \cite{Carter}.) Elliptic subgroups are isomorphic to $\mathbb{S}^1$, so any subgroup isomorphic to $\R$ must be hyperbolic or parabolic. 
%  We first claim any abelian subgroup, $G$, of $SL_2(\R)$ isomorphic to $\R$ must have real weights.  Suppose $G$ has complex weights.  Then $G$ must have a pair of complex weights $z, \bar{z}$  so that the determinant is $1= z \bar{z} = |z|$, or, $z=  \lambda  e^{i \theta}$, and $\bar{z}= \frac{1}{\lambda}e^{-i \theta}$.   This means that $G \cong 
%\big(\begin{smallmatrix}
%\lambda e^{i \theta}&0\\ 0& \frac{1}{\lambda} e^{-i \theta}
%\end{smallmatrix} \big) \cong S^1$.  But we assumed $G \cong \R $, and $\R $ is not compact, so all the weights must be real. \\
%Since $G$ is abelian and hence solvable, by Lie's theorem for Lie algebras (see for example \cite{Hum}, p.16), there is a basis relative to which all of the matrices in the Lie algebra, $\mathfrak{g}$, are upper triangular. Transferring this information to the Lie group via the exponential map, we may choose a basis for $G$ relative to which all of the matrices are upper triangular, with real weights on the diagonal. \\
% If the coordinate-wise projection in the Lie algebra onto the diagonal is 1 dimensional, $G$ is conjugate to the diagonal Cartan subgroup.  If the projection in the Lie algebra onto the diagonal is zero dimensional, then $G$ is  unipotent.  Since these are the all mutually exclusive possibilities for a one-dimensional upper triangular matrix subgroup of $SL_2 ( \R)$, these are all possible limits of the diagonal Cartan subgroup. \\
 
 Conjugate by the sequence of projective transformations  as $n \to \infty$: 
$$ \displaystyle 
 \big(\begin{smallmatrix}
1&n\\ 0& 1
\end{smallmatrix} \big)
\big(\begin{smallmatrix}
a&0\\ 0& \frac{1}{a}
\end{smallmatrix} \big)
\big(\begin{smallmatrix}
 1 &n\\ 0& 1
\end{smallmatrix} \big)^{-1}
=
\big(\begin{smallmatrix}
a &n (a - \frac{1}{a}) \\ 0& \frac{1}{a}
\end{smallmatrix} \big)
\to 
\big(\begin{smallmatrix}
1&t\\ 0& 1
\end{smallmatrix} \big).
$$
Since we want the conjugacy limit to be finite, i.e., we want $n (a - \frac{1}{a})$ to converge to some $t \in \R$, we need $a \to 1$.  Since $a \in \R$ is arbitrary, the limit is a group  where $t$ is any real number.  

The diagonal Cartan subgroup preserves the maximal configuration of 2 fixed points an appreciable distance apart on a projective line.  Applying this sequence of projective transformations identifies the points in the limit, so that the limit configuration consists of one point on a projective line, which is the maximal configuration preserved by the parabolic group. 
 \qed
 \end{proof}

Conjugate the Cartan subgroup (given in the usual basis $\{[1:0], [0:1] \})$, by a hyperreal transformation to change to the basis $\{ [1:0], [1: \delta] \}$. Define \\
$$ \NS G (\delta):=\{ \displaystyle 
 \big(\begin{smallmatrix}
1&1\\ 0& \delta
\end{smallmatrix} \big)
\big(\begin{smallmatrix}
a&0\\ 0& \frac{1}{a}
\end{smallmatrix} \big)
\big(\begin{smallmatrix}
 1 &1\\ 0& \delta
\end{smallmatrix} \big)^{-1}
=
\big(\begin{smallmatrix}
a &\frac{1}{\delta}(a - \frac{1}{a}) \\ 0& \frac{1}{a}
\end{smallmatrix} \big) \}. 
$$
Set $G(\delta) =sh(\textrm{Fin}(\NS{G}(\delta))) \leq SL_2(\R)$.

\begin{corollary}\label{2pt}   If $\delta$ is appreciable, then $G(\delta)$ is hyperbolic.  If $\delta$ is infinitesimal, then $G(\delta)$ is parabolic. 
\end{corollary}

\begin{proof}
%Let $p,q \in \HR P^1$ be distinct points, and define $\NS{G}(p,q) = \{ \NS{A} \in SL_2 ( \HR) | \NS{A}p = p \text{ and } \NS{A}q = q \}$.  Recall $Sh ( \NS{G}(p,q)) = \textrm{Fin}(\NS{G}(p,q))/ \NS{I}$, where $\NS{I}$ is the subgroup of infinitesimal transformations. 
% So, $Sh (\NS{G}(p,q))\cong \{ B \in SL_2 ( \R) | \exists \NS{A} \in \NS{G} (p,q) \text{ such that } \NS{A}-B \text{ is infinitesimal } \}$.    
If $\delta$ is infinitesimal, the finite part of $\NS G(\delta)$ has $a$ is infinitesimally close to 1, so that the upper right entry is finite.  Then $sh(a)=1$, so $G(\delta)$ is conjugate to the parabolic group,  and $G(\delta)$ acts on $\R P^1$ with one fixed point. 

If $\delta$ is appreciable, then $G(\delta)$ is conjugate to a group of hyperbolic projective transformation by a real matrix, and $G(\delta)$ acts on $\R P^1$ with two fixed points. \qed
\end{proof}

\section{Non-Standard Triangles: Proof of Theorem \ref{mainNS}}\label{inf_tri}%%%%%%%%%%%%%%%%%%%%%%%%%%%%%
So far, we have shown every conjugacy limit group of $C$ is conjugate to an element of $\mathcal{Q}$, and determined by a characteristic triangle class in $\mathcal{TQ}$ (theorem \ref{maintri}).  In this section, we establish the bijection between conjugacy limit groups and \emph{equivalence classes of nonstandard triangles}, given as a partition in the table in theorem \ref{mainNS}.   Before proving the table gives a partition, we define $\alpha$.

\phantomsection \label{setup}
A nonstandard triangle may be built from a nonstandard 1-simplex as follows.  Consider $\HR P^2$ with the positive scalar curvature metric inherited from the sphere. This metric is not preserved by projective transformations.   
Three non collinear points, $p,q,x \in \HR P^2$, determine a nonstandard triangle, $\Delta(p,q,x)$.  We may assume they satisfy the following labeling conditions.  The length of the shortest altitude is measured from the point $x$.  Let $\NS H \cong \HR P^1$ be the line containing $p$ and $q$, and $y \in \NS H$ be the foot of the altitude measured from $x$.   Assume $d(y,p) \leq d(y,q)$, and without loss of generality,  assume $p,q,x,y$ have coordinates $p=[1:0], q = [1:\delta], y = [1: \varepsilon] \in \NS H$ and $x=[1:\varepsilon : \eta]$. The shortest altitude is measured from $x$, so $0 \leq |\eta| \leq  |\delta|$ and $0 \leq |\varepsilon| \leq |\delta|$. In the remainder of this section, we denote by $\NS G := \NS G (\delta) \leq SL_2(\HR)$ the group preserving $p$ and $q$, and by $\hat{\NS G} := \hat{\NS G }(p,q,x)\leq SL_3(\HR)$ the group preserving $p,q$ and $x$.  Set $G = sh (\textrm{Fin} ( \NS G))$ and  $\hat G = sh (\textrm{Fin} ( \hat{\NS G}))$

 \begin{center}
 \begin{SCfigure}[][h]
 \includegraphics[scale=0.25]{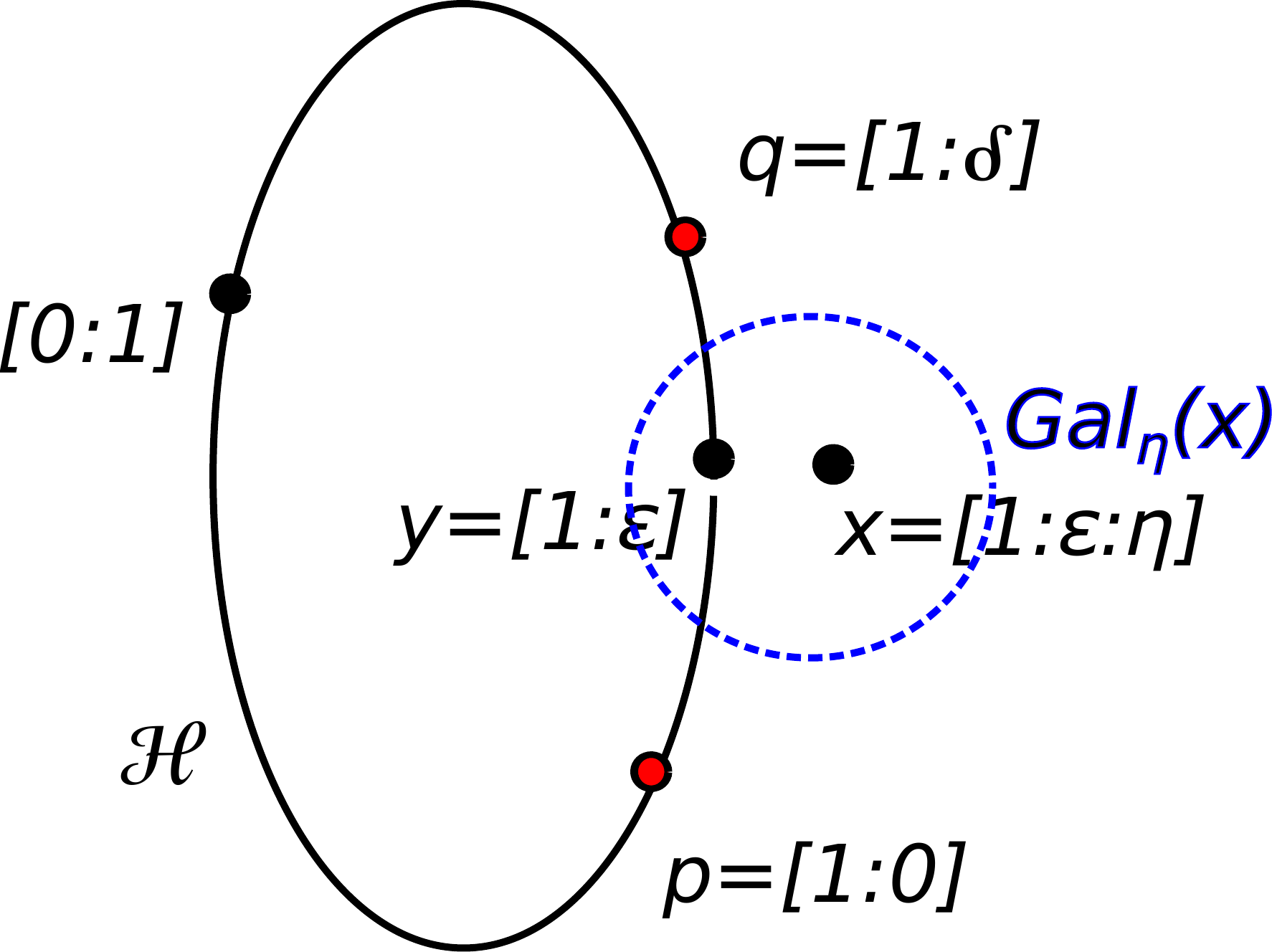}\label{induct}
 \caption{Adding a point to a 1-simplex }
\end{SCfigure}
 \end{center}

\begin{lemma}\label{NSeqrel}
 The table in theorem \ref{mainNS} is a partition on the set of nonstandard triangles, where $\alpha = \frac{\varepsilon \delta}{\eta}$. 
% If $\NS{T}$ and $\NS{S}$ are two nonstandard triangle configurations in different equivalence classes, then $sh(\NS{T}) \neq sh(\NS{S})$ as one of the seven degenerate triangle configurations.
 \end{lemma} 
 \begin{proof}   We show that every nonstandard triangle is in exactly one row of the table. Consider the first two columns of the table: the number of infinitesimal sides and angles. A nonstandard triangle has $0,1$ or $2$ infinitesimal angles, and $0,1$ or $3$ infinitesimal sides.  If a nonstandard triangle has exactly one infinitesimal side (angle), then it must have at least one infinitesimal angle (side).  It is straightforward to see all 7 possibilities for a nonstandard triangle with 0,1 or 2 infinitesimal angles, and 0,1, or 3 infinitesimal sides are listed in the table.  Since $|\delta| \geq  |\eta| $,  if the number of infinitesimal angles and sides is not 2 and 3 respectively, it is easy to check the order of $\frac{\varepsilon \delta}{\eta} $ is determined by the number of infinitesimal sides and angles, these details appear as part of proposition \ref{2to3}.  
 
It remains to show that there is no overlap in the rows.  Examining the first two columns, we see the only repeat is a nonstandard triangle with 3 infinitesimal sides and 2 infinitesimal angles, which  appears three times.  The row for a nonstandard triangle with 2 infinitesimal angles and three infinitesimal sides is determined by whether $\frac{\varepsilon \delta}{\eta}$ is finite, infinitesimal, or infinite. %This is the only repeat of number of infinitesimal sides and angles.   
   %The shadow of a nonstandard triangle with 0 infinitesimal sides and 0 infinitesimal angles is the degenerate triangle configuration $C$. The shadow of a nonstandard triangle with 1 infinitesimal side and 1 infinitesimal angle is the degenerate triangle configuration $F$.  The shadow of a nonstandard triangle with 2 infinitesimal sides and 3 infinitesimal angles is the degenerate triangle configuration $N_1$. 
 %The shadow is determined by the number of infinitesimal sides and infinitesimal angles in a nonstandard triangle.  So it remains only to differentiate between triangles with 2 infinitesimal angles and 3 infinitesimal sides using the ratio $\frac{h}{\theta}$.   All nonstandard triangles with 2 infinitesimal angles and 3 infinitesimal sides have a shadow consisting of one line with one point. 
 \qed
 \end{proof}

%Measure the distance from each point to the nonstandard line containing the other points.  Let $\eta$ be the smallest such distance, and $x$ the point from which the distance $\eta$ is measured.  Let $p,q$ be the other  points of the triangle, $\NS{H}$ the $\HR P^1$ containing $p$ and $q$, and $\delta= \textrm{dist}(p,q)$. Let $y \in \NS{H}$ be the closest point to $x$.  We may assume $y$ is closer to $p$ than $q$, and  $p= [1:0]$. Let $\varepsilon$ be the distance from $p$ to $y$, so $y= [ 1: \varepsilon]$. The \emph{link} of $x$ is $\NS{L}(x) \cong \HR P^1$, the set of all lines through $x$.  Recall $\textrm{Gal}_{\eta}(x) = \{ r \in \HR : x -r \in \eta \cdot \mathbb{F} \}$ from definition \ref{hyperreal}. 
%Also $\NS{G}: = \NS{G} (\delta) \leq \textrm{Aut}(\NS H) \cong PGL_2 (\HR)$, and set $G = sh ( \textrm{Fin}( \NS{G}))$. %Let $\NS{H}_\eta(x) \leq \textrm{Fin}(\NS{G})$ be the subgroup preserving $\textrm{Gal}_{\eta}(x) \cap \NS{H}$. 
 %The action of $\NS{G}$ on $\NS{H}$ is \emph{slowed down}  if $sh(\NS{H}_\eta(x)) \subsetneq sh(\NS{G})$.  This is equivalent the action on the orbit closure $sh(\NS H) \subset \R P^2$ being decompressed in the limit $sh(\hat{\NS G}( p,q,\left[ \begin{smallmatrix} 0\\0\\1 \end{smallmatrix}\right])) \to sh (\hat{\NS G} ( p,q,x))$.
Recall a projective transformation is appreciable if it maps an appreciable basis to an appreciable basis.   Let $\NS{X} \subset \HR P^n$ be a subspace, and extend the usual basis of $\NS{X}$ to the usual basis of $\HR P^n$.  We say a group $\NS{ F}$ \emph{acts finitely} on $\NS{X}$, if $f |_{\NS{X}}$ is a finite transformation, for all $f \in \NS{F}$. 

Let $[v] \in \HR P^n$.    The \emph{link} of $[v]$ is $\NS{L}(v) \cong \HR P^{n-1}$, the set of all lines through $[v]$. Given a projective basis $\{ [e_0] ,[ e_1], ...,[ e_{n+1}] \}$ for $\HR P^n$ with $[e_0]=[v]$, a projective basis for $\NS{L}(v) \cong \HR P^{n-1}$ consists of lines $\{ \langle{ [v],[e_i]}\rangle | 1 \leq i \leq n+1 \}$. Recall a basis, $\NS{B}$, for $\NS{L}(v)$ is appreciable if $sh(\NS{B})$ is a basis for $L(v): = sh(\NS{L}(v))$. 
In $\HR P^2$ a basis for $\NS L (v)$ is appreciable if and only if the angles between the projective lines in the basis for $\NS{L}(v)$ are appreciable.

 Let $\NS{G}= \NS G( \delta) \leq SL_2 (\HR)$ act on $\NS{H}$, and let $x \in \HR P^2 - \NS H$ as in figure \ref{induct}. The action of $\NS{G}$ on $\NS{L}(x)$ is defined as follows.  In projective space, every pair of lines intersect in a point, so every line in $\NS{L}(x)$ intersects $\NS{H}$ in a point, $z$.  Thus $\NS{L}(x) = \{ \langle{x,z}\rangle : z \in \NS{H} \}$, and the action of $\textrm{Fin}(\NS{G})$ on $\NS{L}(x)$ is given as $\langle{x,z}\rangle \mapsto \langle{x, g(z)}\rangle$, for $g \in \NS{G}$.  %As before, we say $\NS{G}$ \emph{acts finitely} on $\NS{L}(x)$ if the image of an appreciable basis is an appreciable basis.   
 We will show the conjugacy limit group that preserves the shadow of the nonstandard triangle is controlled by the action on the link of the new point, $\NS{L}(x)$. 
 
  Recall $\textrm{Gal}_{\eta}(x) = \{ r \in \HR : x -r \in \eta \cdot \mathbb{F} \}$ from definition \ref{hyperreal}. 

\begin{lemma}\label{link1} $\textrm{Fin}(\NS{G})$ preserves $\textrm{Gal}_{\eta}(y)$ if and only if $\textrm{Fin}(\NS{G})$ acts finitely on $\NS{L}(x)$. 
\end{lemma}
 \begin{proof}  

Let  $\{ v_1= [1: \varepsilon + \eta], v_2=[1: \varepsilon -\eta], v_3= [1: \varepsilon] \}$ be a basis for $\NS H$,  and set $\NS{B}= \{\langle{ x, v_i}\rangle : i=1,2,3\}$, a basis for $\NS{L}(x)$.   The basis $\NS{B}$ is appreciable since the lines $\langle {x,v_1}\rangle$ and $\langle{x,v_3}\rangle$ form a $45^{\circ}$ isosceles triangle with $\NS{H}$, and $\langle{x,v_2}\rangle$ is perpendicular to $\NS H$. The distance a point $z \in \NS H$ is moved  by $g \in \NS G$ is $|z -  g.z| $.  The distance a point is moved in $\NS{L}(x)$ is $\angle (\langle{x, z} \rangle, \langle{x, (g.z)})\rangle \approx \frac{ | z - g.z |}{\eta}$.  

   So $\textrm{Fin}(\NS{G})$ acts finitely on $\NS{L}(x)$ if and only if $\textrm{Fin}(\NS{G})$ acts finitely on $\NS{B}$, the basis for $\NS L(x)$.  But $\textrm{Fin}(\NS{G})$ acts finitely on $\NS B$ if and only if $\textrm{Fin}(\NS{G})$ keeps the angles between the lines in $\NS B$ appreciable, i.e. if $\textrm{Fin}(\NS{G})$ moves the points $v_i \in \NS H$ a distance of at most order $\eta$.  Thus $\textrm{Fin}(\NS{G})$ acts finitely on $\NS L (x)$ if and only if $\textrm{Fin}(\NS{G})$ preserves $\textrm{Gal}_{\eta}(y)$.  \qed
 \end{proof} 
  
Two nonzero hyperreals $\alpha, \beta \in \HR$ have the same \emph{order} if and only if $\frac{\alpha}{\beta}$ is appreciable.  We denote this $\alpha \approx \beta$. 

\begin{lemma}\label{link} 
%Let $p,q  \in \NS{H}=  \HR P^1$ be distinct points. Let $x \in \HR P^2 - \NS{H}$. If $\eta$ is infinitesimal, and $\mathcal{G}_{\eta}(x)$ contains $p$ or $q$, then the action of $\NS{G}$ on $lk(x) \subset \HR P^2$ is finite.  If $\mathcal{G}_{\eta}(x)$ contains neither $p$ or $q$,  the action of $\NS{G}$ on $lk(x)$ is infinite.  In particular, 

%Let $\NS G := \NS G (\delta) \leq SL_2 (\HR)$. Then 
$\textrm{Fin}(\NS{G})$ moves a point in $\textrm{Gal}_{\eta}(y)$ a distance of at most order $\varepsilon \delta$. 
\end{lemma} 
\begin{proof}  

Recall
$$ \NS{G}(\delta)=\{ \displaystyle 
\big(\begin{smallmatrix}
1&1\\ 0& \delta
\end{smallmatrix} \big)
\big(\begin{smallmatrix}
a&0\\ 0& \frac{1}{a}
\end{smallmatrix} \big)
\big(\begin{smallmatrix}
 1 &1\\ 0& \delta
\end{smallmatrix} \big)^{-1}
=
\big(\begin{smallmatrix}
a &\frac{1}{\delta} (a - \frac{1}{a}) \\ 0& \frac{1}{a}
\end{smallmatrix} \big) \}. 
$$
 In $\textrm{Fin}(\NS{G})$,  we have $\frac{1}{\delta} (a - \frac{1}{a}) = 2t$, a finite hyperreal.  The action on $y= [ 1: \varepsilon]$ depends on $\eta, \delta, \varepsilon$, and we want to find the subgroup of $\textrm{Fin}(\NS{G})$ that preserves $\textrm{Gal} _{\eta}(y)$.  We have:
$$\big(\begin{smallmatrix}
a &\frac{1}{\delta} (a - \frac{1}{a}) \\ 0& \frac{1}{a}
\end{smallmatrix} \big)
[1: \varepsilon] = [ a + \frac{\varepsilon}{\delta} (a - \frac{1}{a}) : \frac{\varepsilon}{a} ] = [ 1: \frac{\frac{\varepsilon}{a} }{a+ \frac{\varepsilon}{\delta} (a - \frac{1}{a})}] .$$
%Solving the quadratic equation relating $a, \delta ,t$, gives $ a = \delta t + \sqrt{ 1 + ( \delta t ) ^2}$, and $a^2 = 2 a \delta t +1$.   
We want to find the distance $y$ is moved.  Using $ \frac{1}{a} - a = 2 t \delta$ we get
$$\big| \frac{\frac{\varepsilon}{a} }{a+ \frac{\varepsilon}{\delta} (a - \frac{1}{a})} - \varepsilon \big|=  \varepsilon 
\big|  \frac{\frac{1}{a} - a + 2 t \varepsilon}{a + 2 t \varepsilon} \big|= 2 t \varepsilon \big| \frac{\delta + \varepsilon}{a + 2t \varepsilon} \big|  \approx 2 t \varepsilon | \delta + \varepsilon | $$
%\frac{\varepsilon(1 - 2 a \delta t -1 - 2 a \varepsilon t)}{2 a \delta t + 1 + 2 a \varepsilon t}= \frac{\varepsilon(2 a \delta t - 2 a \varepsilon t)}{2 a \delta t + 1 + 2 a \varepsilon t } = \frac{\varepsilon (\varepsilon - \delta)}{\delta + \varepsilon - \frac{1}{2 t a}}$$
%$$\varepsilon \mapsto \frac{\frac{\varepsilon}{a}}{a + 2 t \varepsilon} = \frac{\varepsilon}{a^2 + 2 a t \varepsilon} = \frac{\varepsilon}{2 a \delta t + 2 a t \varepsilon}   =    \frac{\varepsilon}{4 t^2 ( \delta ^2 +    \varepsilon \delta)} . $$
%$$\textrm{So,    } \varepsilon \mapsto \frac{\frac{\varepsilon}{a}}{a+ \frac{\varepsilon}{\delta}(a- \frac{1}{a})} = \frac{\varepsilon}{\frac{\varepsilon}{\delta}(a^2-1)} \approx \frac{\varepsilon}{\frac{\varepsilon}{\delta}},$$
 %Thus $y$ is moved $\varepsilon a t ( 2 \delta + \varepsilon)$.  
 since $a \approx 1$ in $\textrm{Fin}(\NS G)$. 
 % and the last one since $\delta \geq \varepsilon \geq 0$. Thus $y$ is moved a distance of order $\varepsilon \delta$.
Since $0 < \varepsilon \leq \delta$ then $ \varepsilon \delta \leq \varepsilon ( \delta + \varepsilon) \leq 2 \varepsilon \delta$.  Therefore 
$$ 2 t \varepsilon | \delta + \varepsilon | \approx \varepsilon(\delta + \varepsilon)\approx \varepsilon \delta,$$
and $y$ is moved a distance of order $\varepsilon \delta$. 
 \qed
\end{proof}

\begin{corollary}\label{linkfin} $\textrm{Fin}(\NS{G})$ acts finitely on $\NS{L}(x)$ if and only if $\frac{\varepsilon \delta}{\eta}$ is finite.  Moreover, the action of $\textrm{Fin}(\NS{G})$ on $\NS{L}(x)$ is infinite if $\frac{\varepsilon \delta}{\eta}$  is infinite, and $ sh(\textrm{Fin}( \NS{G}))$ acts as the identity on $L(x)$ if $\frac{\varepsilon \delta}{\eta}$ is infinitesimal. 
\end{corollary}

\begin{proof} By lemma \ref{link},  the distance $y \in \textrm{Gal} _{\eta}(x)$  is moved in $\NS{H}$ is of order $\varepsilon \delta $, and so by the proof of lemma \ref{link1},  $\textrm{Fin}(\NS{G})$ moves points in $\NS{L}(x)$ a distance of order $\frac{\varepsilon \delta}{\eta}$. %There is a homomorphism:
%\begin{center}
%\begin{tikzpicture}
%\node(A){ $\theta: \textrm{Fin}(\NS{G})$};
%\node[right of =A, node distance= 2cm](B){$\textrm{Fin}(\NS{G}),$};
%\node[ below of = A , node distance= 1cm](C){$\NS{H}$};
%\node[below of = B , node distance = 1 cm] (D){ $\NS{L}(x)$};
%\node[ right of =B, node distance= 3cm](E) {where $ \theta(A) = \frac{\varepsilon\delta}{\eta} A$.};
%\path(C) edge [loop above,  style ={-stealth}]  (C); 
%\path(D) edge [loop above,  style ={-stealth}]  (D); 
%\path(A) edge [ style ={-stealth}] (B); 
%\end{tikzpicture} 
%\end{center}
% The image of $\theta$ is finite if and only if $\frac{\varepsilon \delta}{\eta}$ is finite.  If $\frac{\varepsilon \delta}{\eta}$ is infinitesimal, then $sh(\textrm{Im}(\theta))$ is the identity. 
The result follows.  \qed
\end{proof}

\begin{figure}[h]
        \centering
        \begin{subfigure}[b]{0.3\textwidth}
                \includegraphics[width=\textwidth]{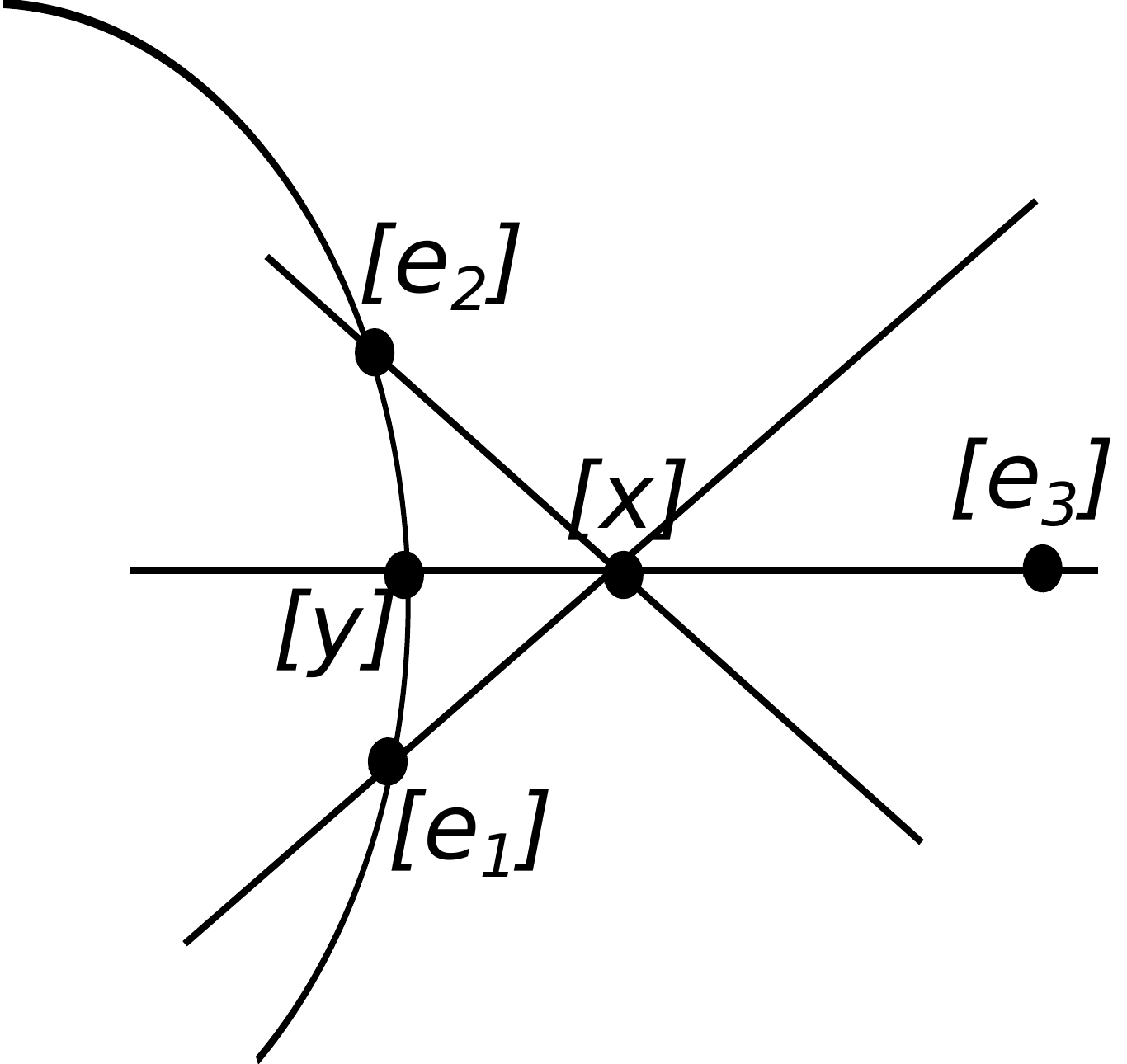}
                \caption{Choosing $\NS{B}$}
                \label{fig:gull}
        \end{subfigure}%
\hspace{.5in}
        \begin{subfigure}[b]{0.35\textwidth}
                \includegraphics[width=\textwidth]{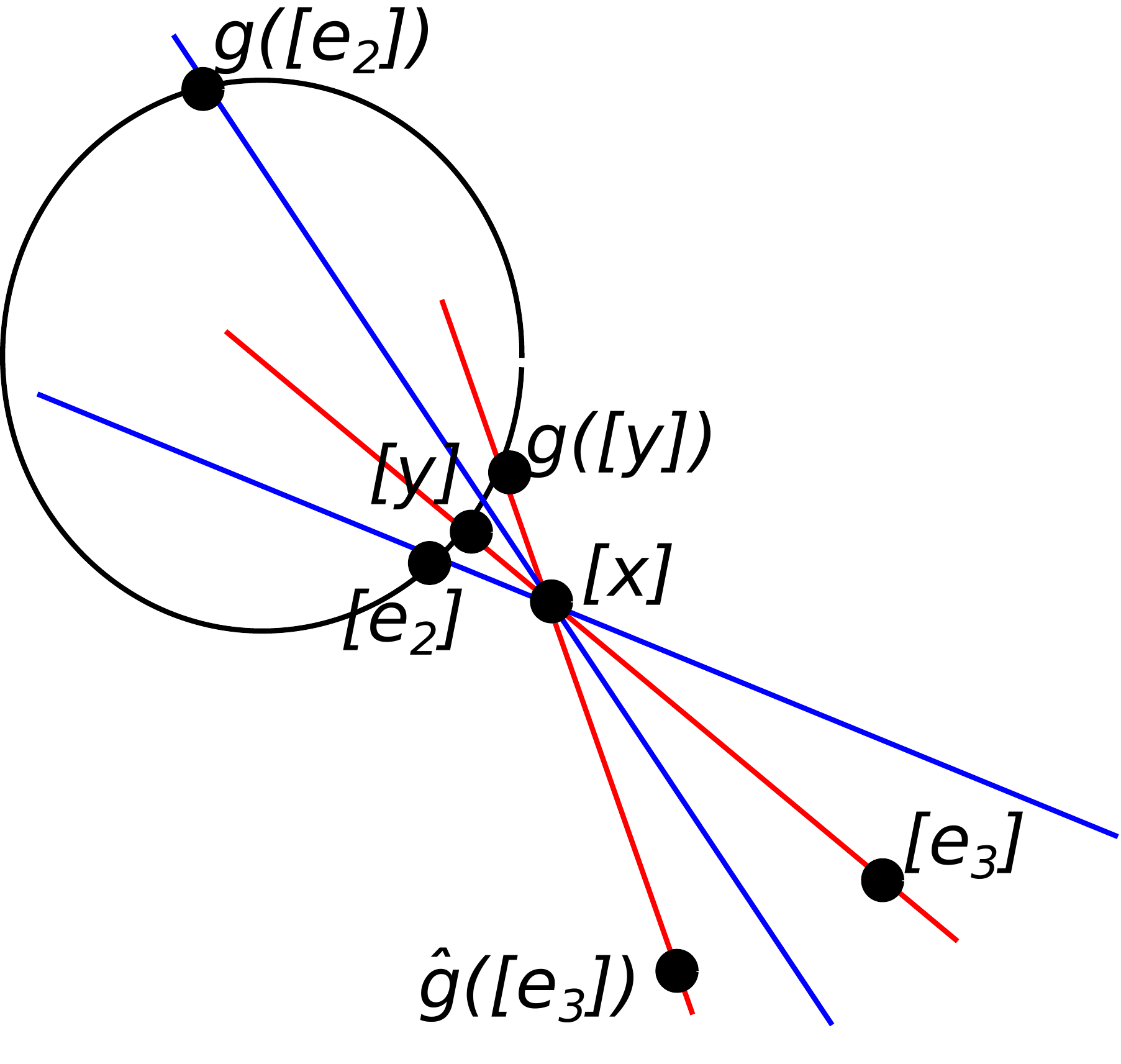}
                \caption{The action on $\NS{L}(x)$ extended to $\HR P^2$. See the image of $\NS B$ under $\hat{\NS{G}}$.}
                \label{fig:tiger}
        \end{subfigure}
        ~ %add desired spacing between images, e. g. ~, \quad, \qquad, \hfill etc.
        \caption{Extending $\NS{G}$ to $\hat{\NS{G}}$}\label{fig:animals}
\end{figure}

We have shown the ratio $\frac{\varepsilon \delta}{\eta}$ determines the action of $\textrm{Fin}(\NS{G})$ on $\NS{L}(x)$.  Next we show how to extend $\NS{G} \leq SL_2 (\HR)$ to $\hat{\NS{G}} \leq SL_3 (\HR)$. 

\begin{lemma}\label{link_ext} Let $\NS{G}= \NS G (\delta) \leq SL_2 ( \HR)$ act on $\NS{H}$. If $\textrm{Fin}( \NS{G})$ acts finitely on $\NS{L}(x)$, define $\NS{\hat{G}} \leq SL_3 (\HR)$ to be the set of elements which fix $x$, act finitely on $\HR P^2$, and for every $\hat {g} \in \NS{\hat{G}} $, the restriction $\hat{g} | _ { \NS{H}} = g $ for some $g \in \NS{G}$.   Then $\hat{G}: = sh(\textrm{Fin}(\hat{\NS{G}})) $ is a conjugacy limit group, and the action of $\hat{G} $ on $H$ coincides with the action of $G$ on $H$. 
\end{lemma} 

\begin{proof} 
We pick a projective basis $ \NS{B}= \{[e_1],[ e_2], [e_3],[ e_1 + e_2 + e_3]\}$ for $\HR P^2$, and show the image of $\NS{B}$ under $\NS{\hat{G}}$ is an appreciable basis.  By assumption, $\NS{G}$ fixes $[e_1] \in \NS{H}$.   Let $[y] \in \NS{H}$ be the point closest to $[x]$.  Choose $[e_3] \in \langle {[x],[y]}\rangle$ an appreciable distance from $[e_1]$.  Then the lines $\langle {[x],[e_1]} \rangle$ and $\langle{[x],[e_3]}\rangle= \langle{[x],[y]}\rangle$ are at an appreciable angle, since $[x]$ is infinitesimally close to $\NS{H}$.  Pick $[e_2] \in \NS{H}$ an appreciable distance from $[e_1]$ and $[e_3]$, and such that $\langle [x], [e_2] \rangle$ and $\langle {[x], [e_3]}\rangle$ are at an appreciable angle. Again, $\langle{[x],[e_2]}\rangle$ and $\NS{H}$ are at an appreciable angle, since $[x]$ is infinitesimally close to $\NS{H}$. 
%$\textrm{Fin}(\NS{G})$ acts finitely on $\NS{H}$, so 

If $g \in \textrm{Fin}(\NS{G})$, then $g ([e_2])$ is an appreciable distance from $g([e_1]) =[e_1]$.   Let $\hat{g} \in \NS{\hat{G}}$, so that $\hat{g}|_{\NS{H}} = g$.  Since $\hat{g}([e_3])$ lies on the line $\langle {[x],g([y])} \rangle$, there is a 1-parameter hyperreal family of choice of image of $[e_3]$.  Choose $\hat{g}([e_3])$ to be an appreciable distance from $\hat{g}([e_1]) = [e_1]$.  The action of $\textrm{Fin}(\NS{G})$ on $\NS{L}(x)$ is finite, so $\langle{[x], [\hat{g}(e_i)]}\rangle$ and $\langle{[x], \hat{g}([e_j])}\rangle$ are an appreciable distance apart, for all $i \neq j$.   Thus the image of an appreciable basis is an appreciable basis, and $\{p,q,x\}$ is the image of the usual basis under a nonstandard projective transformation.  By theorem \ref{HRlimgp}, $sh(\textrm{Fin}(\hat{\NS{G}})) $ is a conjugacy limit group, and by construction $\NS{ \hat{G}} | _ {\NS{H}}$ is isomorphic to a subgroup of  $ \NS{G}$.  \qed
\end{proof}

% Recall $\delta $ is the length of the longest side of the nonstandard triangle, and $\eta$ is the smallest altitude. Together, $\delta$ and $\eta$ determine the number of infinitesimal sides and angles.  For small angles, $\sin(\theta) \approx \theta$, so $\frac{\eta}{\varepsilon}$ is approximately the largest infinitesimal angle. The ratio $\frac{\varepsilon \delta}{\eta}$ is the ratio of the longest side to the largest infinitesimal angle, as in theorem \ref{mainNS}. 

By theorem \ref{maindig} there are 5 conjugacy classes of conjugacy limit groups, and by theorem \ref{HRlimgp} every conjugacy limit group of $C$ is $sh(\textrm{Fin} ( \NS P \NS C \NS P^{-1}))$.  We show next that the conjugacy class of any conjugacy limit group is completely determined by the hyperreal numbers $\delta, \eta$ and $\frac{\varepsilon \delta}{\eta}$.  Recall $0 \leq |\eta| \leq | \delta|$, and $0 \leq |\varepsilon| \leq |\delta|$.

\begin{proposition}\label{2to3}  Given $p,q,x \in \HR P^2$ in general position, let $\hat {\NS G} (p,q,x) \leq SL_3 (\HR)$ be the group preserving $p,q,x$.  Set $\hat{G} = sh (\hat {\NS G} (p,q,x))$.
%there are 5 groups up to conjugacy, $\hat{G} := sh(\textrm{Fin}(\NS{\hat{G}} (p,q,x)))  \leq SL_3(\R)$, such that $\hat{G}|_{H}$ is isomorphic to a subgroup of  $G = sh (\textrm{Fin}( \NS{G} (p,q))) \leq SL_2 (\R)$. 
\begin{enumerate} 
\item If $\delta$ is appreciable, and  
\begin{enumerate} 
\item $\eta$ is appreciable, then $\hat{G}$ is conjugate to $C$ and $\frac{\varepsilon \delta}{\eta}$ is appreciable 
\item  $\eta$ is infinitesimal, and $p \in \textrm{Gal}_{\eta}(x)$,  then  $\hat{G}$ is conjugate to $F$ and  $\frac{\varepsilon \delta}{\eta}$ is finite
\item  $\eta$ infinitesimal, but $p \not \in \textrm{Gal}_{\eta}(x)$, then  $\hat{G}$ is conjugate to $N_3$ and $\frac{\varepsilon \delta }{\eta}$ is infinite 
\end{enumerate} 
\item If $\delta$ is infinitesimal, then $ \eta$ is infinitesimal, and 
\begin{enumerate} 
 \item if $\frac{\varepsilon \delta }{  \eta}$ is infinitesimal, then $\hat{G}$ is conjugate to $N_2$.  
 \item  if $\frac{\varepsilon \delta}{  \eta}$is appreciable, then $\hat{G}$ is conjugate to $N_1$. 
 \item  if $\frac{\varepsilon \delta}{  \eta}$ is infinite, then $\hat{G}$ is conjugate to $N_3$. 
 \end{enumerate} 
 \end{enumerate}
\end{proposition}

\begin{proof}
\begin{description}
\item[Case 1a:]  The points are an appreciable distance apart, and each point is an appreciable distance from the hyperplane containing the other two points.  Since $\eta$ is appreciable, $\textrm{Gal}_{\eta}(x) \cap \NS{H }=\NS{H}$, so the action of $\textrm{Fin}(\NS{G})$ preserves $\textrm{Gal}_{\eta}(x)$, and $\NS G$ acts finitely on $\NS L(x)$. The action of $\hat{G}$ on each of the three projective lines is hyperbolic.  Thus $\hat{G}$ is conjugate to $C$, since this is the only group in $\mathcal{Q}$ which acts hyperbolically on more than one line.  

Since $\delta, \varepsilon $ and $\eta$ are all appreciable, then
$\frac{\varepsilon \delta}{\eta}$ is appreciable. 

 \item[Case 1b:]  Suppose $p \in \textrm{Gal}_{\eta}(x)$.   Then $p$ and $x$ are infinitesimally close, and $q$ is an appreciable distance away from $\langle p, x \rangle$. So, $sh(\Delta(p,q,x))$ has two distinct points, $sh(p)=sh(x)$ and $sh(q)$, and two distinct lines, $sh(\langle p,q \rangle) =sh(\langle x,q \rangle)$ and $sh(\langle p, x \rangle)$.  The action of $\hat{G}$ on $sh(\langle{p,x}\rangle)$  is parabolic, since $sh(\langle{p,x}\rangle)$ contains a unique fixed point.  The action of $\hat{G}$ on $H= sh(\langle p, q \rangle)$ is hyperbolic, since $H$ contains two distinct fixed points.    Thus $\hat{G}$ has two 0 dimensional invariant subspaces, and two 1 dimensional invariant subspaces, one with a parabolic action, and one with a hyperbolic action. So, $\hat{G}$ is conjugate to $F$. 
 
 Since $\delta$ is appreciable, $\frac{\varepsilon \delta}{\eta} \approx \frac{\varepsilon}{\eta}$.  Since $p \in \textrm{Gal}_{\eta}(x)$, then $ | \varepsilon | \leq | \eta|$, which implies $\frac{\varepsilon}{\eta}$ is finite. 
 
\item[Case 1c:] Since $\delta $ is appreciable, $\frac{\varepsilon \delta} {\eta} \approx \frac{\varepsilon}{\eta}$.  Since $\eta$ is infinitesimal, and $\textrm{Gal}_{\eta}(x)$ does not contain $p$, then $| \varepsilon|$ has larger order than  $| \eta|$, which implies $\frac{\varepsilon}{\eta}$ is infinite.   

Since  $\frac{\varepsilon \delta}{\eta}$ is infinite, then the action of $\textrm{Fin}(\NS{G})$ on $\NS{L}(x)$ is infinite by corollary \ref{linkfin}.  The subgroup of $\textrm{Fin}(\NS{G})$ which acts finitely on $\NS{L}(x)$ is infinitesimal, and its shadow in $G$ is the identity.  Thus $\hat{G}$ is conjugate to $N_3$,  since this is the only group in $\mathcal{Q}$ which acts as the identity on a line.

\item[Case 2a:] By assumption $\delta$ and $\eta$ are infinitesimal, so the points $p,q,x$ are infinitesimally close. Therefore $sh(\Delta(p,q,x))$ has one point, $sh(p)=sh(q)=sh(x)$, which is the fixed point under the action of $\hat{G}$. By lemma \ref{unip} which follows this proof, $\hat{G}$ is unipotent, and so $ \hat{G}$ acts parabolically on any line it preserves through $sh(p)$.  Since $ \frac{\varepsilon \delta }{  \eta}$ is infinitesimal,   $\hat{G}$ acts as the identity on $L(x)$ by corollary \ref{linkfin}, and $\hat{G}$ preserves at least two 1 dimensional invariant subspaces.  There are two groups in $\mathcal{Q}$ with a single fixed point, $N_1$ and $N_2$.  But, $N_1$ preserves only one line, and $\hat{G}$ preserves at least two lines, so $\hat{G}$ is conjugate to $N_2$.  

\item[Case 2b:] Since $\frac{\varepsilon \delta }{  \eta}$  is appreciable, the action of $\textrm{Fin}(\NS{G})$ on $\NS{L}(x)$ is finite by corollary \ref{linkfin}, and $G$ does not fix $L(x)$ point wise.  By lemma \ref{link_ext}, the action of $\hat{G}$ on $H$ coincides with the nontrivial action of  $G$ on $H$. Since $\delta$ and $\eta$ are infinitesimal, the points $p,q,x$ are infinitesimally close,  so $sh(\Delta(p,q,x))$ has one point. By lemma \ref{unip}, $\hat{G}$ acts parabolically  on $H$, since $\hat{G}$ fixes a single point.  Since $\hat{G}$ does not fix $L(x)$,  then $\hat{G}$ has only one 1 dimensional invariant subspace and one 0 dimensional invariant subspace,  so $\hat{G}$ is conjugate to $N_1$.  

  \item[Case 2c:] Since $\frac{ \varepsilon \delta }{\eta}$ is infinite, then $\NS{G}$ acts infinitely on $\NS{L}(x)$ by corollary \ref{linkfin}.   The subgroup of $\textrm{Fin}(\NS{G})$ which acts finitely on $\NS{L}(x)$ is infinitesimal, and its shadow in $G$ is the identity.  Thus $\hat{G}$ is conjugate to $N_3$,  since this is the only group in $\mathcal{Q}$ which acts as the identity on a line. \qed
  %The cases in this theorem are mutually exclusive, and cover all possible nonstandard triangles.  These are all possible limit groups by lemma \ref {2dim} and the proof of the first part of theorem \ref{main}. 
  \end{description}
\end{proof} 

Recall from lemma \ref{uppertri} that every conjugacy limit group is conjugate to one under a sequence of upper triangular sequence of matrices, and by theorem \ref{HRlimgp} that any conjugacy limit of $C \leq SL_3(\R)$ is $\hat{G} : = \textrm{Sh}( \textrm{Fin}(\NS{P} \NS C \NS{P}^{-1}))$, where 
\begin{equation}\label{NSP}
 \NS{P}=  \left(
\begin{array}{ccc}
1 & 1 & 1 \\
0 & \delta & \varepsilon \\
0 & 0 & \eta \end{array}
\right). \end{equation}

\begin{lemma}\label{unip} If $\delta$ is infinitesimal, then $\hat{G} = sh (\textrm{Fin} ( \hat { \NS G} (p,q,x)))$ is conjugate to a unipotent group. 
\end{lemma}

\begin{proof}  Since $\delta$ is infinitesimal, and $| \delta| \geq | \eta|$ and $| \delta | \geq | \varepsilon|$, then $\eta$ and $\varepsilon$ are infinitesimal.  %The group $\hat{G}$ is conjugate to $sh (\textrm{Fin}( \NS{ PCP}^{-1}))$. 
It is easy to check that all of the weights must be $1$ in $sh (\textrm{Fin}( \NS{ PCP}^{-1}))$. %when $\delta, \varepsilon$ and $\eta$ are infinitesimal. 
\end{proof} 

  This concludes the proof of theorem \ref{mainNS}, by establishing a bijection between equivalence classes of nonstandard triangles and conjugacy limit groups.  We reproduce a table here for completeness.   The ratio $\frac{\varepsilon \delta}{\eta} = \alpha$ in theorem \ref{mainNS}. 
  
  $$\begin{array}{|c|c|c||c|}
  \hline
  \delta & \eta & \frac{\varepsilon\delta}{\eta} & \textrm{Group} \\
  \hline 
  \textrm{appreciable} &   \textrm{appreciable} &     \textrm{appreciable} & C \\
   \textrm{appreciable}& \textrm{infinitesimal} &   \textrm{finite} & F \\
     \textrm{infinitesimal}&   \textrm{infinitesimal}  &  \textrm{appreciable}  &  N_1 \\ 
    \textrm{infinitesimal}&   \textrm{infinitesimal}  &  \textrm{infinitesimal}  &  N_2 \\ 
\textrm{finite}&   \textrm{infinitesimal}  &  \textrm{infinite}  &  N_3 \\
    \hline
  \end{array} 
  $$
%Notice the ratio $\frac{\varepsilon \delta}{\eta}$ does \emph{not} change the shadow of the nonstandard triangle, but it does determine which conjugacy limit group in $\mathcal{Q}$ preserves the shadow of the nonstandard triangle. 
The shadow of $\NS T$ does not depend on $\frac{\varepsilon \delta}{\eta}$, but $sh (\textrm{Fin} (\hat{\NS  G} (p,q,x))$ does depend on $\frac{\varepsilon \delta}{\eta}$.

%\begin{cor}  Let $\NS{T}$ be a nonstandard triangle, $\NS{\hat{G}}$ be the nonstandard diagonal Cartan conjugate preserving $\NS{T}$, and $\hat{G} = \textrm{Sh}(\textrm{Fin}( \NS{\hat{G}})))$ be the limit group preserving $\textrm{Sh}(\NS{T})$.

%\begin{enumerate}

%\item If all angles and sides of $\NS{T}$ are appreciable, $\hat{G}$ is conjugate to $C$. 

%\item If $\NS{T}$ has 1 infinitesimal angle, and 1 infinitesimal side, $\hat{G}$ is conjugate to $F$. 

%\item If $\NS{T}$ has 0 or 1 infinitesimal angles, and all infinitesimal sides, $\hat{G}$ is conjugate to $N_2$. 

%\item Suppose $\NS{T}$ has 2 infinitesimal angles. Then there is only 1 invariant subspace and one weight, so $\hat{G}$ is conjugate to  $N_1$, $N_2 $, or $N_3$.  Depending on whether $ \frac{h}{\theta} $ is infinite, finite, or infinitesimal; $\hat{G}$ acts parabolically ($N_1$ or $N_2$), or as the identity ($N_3$) on this invariant subspace. 
%\end{enumerate} 
%\end{cor}

\begin{corollary}\label{coords} %Given a hyperreal conjugating matrix 
%$$ \NS{P}=  \left(
%\begin{array}{ccc}
%1 & 1 & 1 \\
%0 & \delta & \varepsilon \\
%0 & 0 & \eta \end{array}
%\right), $$
The columns of $\NS{P}$ in (\ref{NSP}) are the coordinates of the vertices of a nonstandard triangle.  The equivalence class of nonstandard triangle determines the conjugacy limit group $\hat{G} $,
%$= \textrm{Sh}( \textrm{Fin}(\NS{P} C \NS{P}^{-1}))$, or in $SL_3(\R)$,
 the conjugacy limit of $C$ under the sequence $P_n \leq GL_3(\R)$ with sequences in the strictly upper triangular portion, $\delta_n,  \epsilon_n, \eta_n,$ converging to $\delta,\varepsilon , \eta$.  By proposition \ref{2to3}, $\hat{G}$ depends only on the relative orders of $\delta$, $\eta$, and $\frac{\varepsilon \delta }{\eta}$. 

%$$\begin{array}{|c|c|c||c|} 
%\hline
%\delta & \eta & \frac{\varepsilon \delta}{\eta} & \textrm{Conjugacy Limit} \\
%\hline 
%\textrm{appreciable} & \textrm{appreciable} & \textrm{appreciable} & C\\ 
%\textrm{appreciable} & \textrm{infinitesimal}&  \textrm{appreciable} & F \\
%\textrm{infinitesimal} &\textrm{infinitesimal} & \textrm{appreciable} & N_1 \\
%\textrm{infinitesimal} & \textrm{infinitesimal} &  \textrm{infinitesimal} & N_2 \\
%\ast  & \textrm{infinitesimal} & \textrm{infinite} & N_3 \\
%\hline
%\end{array} $$

\end{corollary}

\section{Higher Dimensions} %%%%%%%%%%%%%%%%%%%%%%%%
Haettel shows in $SL_4(\R)$ that the space of 3 dimensional abelian groups is equal to the space of conjugacy limits of the diagonal Cartan subgroup.   Are these two spaces the same in higher dimensions?   Haettel has shown they are the same for $n =3,4$. For $n=5,6$, the author gives examples in \cite{LeitnerSLn} of $n-1$ dimensional abelian subgroups of $SL_n(\R)$ which are not conjugacy limits of the diagonal Cartan subgroup.   When $n \geq 7$, Haettel, and Iliev and Manivel show by dimension count that there are $n-1$ dimensional abelian subgroups of $SL_n(\R)$ which are not conjugacy limits of the Cartan subgroup.   In fact, when $n \geq 7$, the author shows in \cite{LeitnerSLn} that there are infinitely many conjugacy classes of conjugacy limits of the diagonal Cartan subgroup. 

Conjugacy limits of the diagonal Cartan subgroup may be studied in higher dimensions using  nonstandard simplices.  Again, we build a nonstandard $n$-simplex, $\NS{T}_n$ by adding a point, $x$, to a nonstandard $(n-1)$-simplex, $\NS{T}_{n-1}$.  Let $\NS{G}$ be the group that preserves $\NS{T}_{n-1}$. The group, $\hat{\NS{G}}$, that preserves $sh(\NS{T}_n)$ is the standardization of the finite part of the extension group, analogous to lemma \ref{link_ext}.  Let $\eta$ be the distance in $\HR P^n$ from $x$ to $\NS{T}_{n-1}$, and $\NS{H}$ be the highest dimensional cell of $\NS{T}_{n-1}$ such that $\textrm{Gal}_n (x) \cap \NS{H} \neq \emptyset$. Then $\hat{\NS{G}}$ is an extension of the subgroup of $\NS{G} \curvearrowright \NS H$ that acts finitely $\NS{L}(x)$. 

%%%%%%%%%%%%%%%%%%%%%%%%%%%%%
\section*{Acknowledgements} 
The author thanks \emph{Daryl Cooper} for many insightful conversations, patience, and for suggesting the use of the hyperreals. 
The author thanks \emph{Darren Long}, \emph{David Dumas}, and \emph{Jeff Danciger} for many helpful discussions.  The referee also provided some excellent ideas for restructuring the paper, and improving the clarity of some of the arguments. 

The author was partially supported by NSF grants  DMS--0706887,  1207068 and  1045292.
The author spent fall 2013 at ICERM, and had many illuminating discussions with the other visiting academics there. 
The author acknowledges support from U.S. National Science Foundation grants DMS 1107452, 1107263, 1107367 {\em RNMS: GEometric structures And Representation varieties} (the GEAR Network).

%%%%%%%%%%%%%%%%%%%%%%%%%%%%%%%%%%%%%%

\end{document}